\documentclass{amsart}

\pdfoutput=1
\usepackage{amssymb,microtype,fullpage,hyperref,enumerate,graphicx,nicefrac,xfrac,bm}

\hypersetup{colorlinks=true, citecolor=blue, linkcolor=blue, urlcolor=blue, pdfstartview=FitH, pdfauthor=Mikolaj Fraczyk and Gergely Harcos and Peter Maga and Djordje Milicevic, pdftitle=The density hypothesis for horizontal families of lattices}

\newcommand{\NN}{\mathbb{N}}
\newcommand{\ZZ}{\mathbb{Z}}
\newcommand{\QQ}{\mathbb{Q}}
\newcommand{\RR}{\mathbb{R}}
\newcommand{\CC}{\mathbb{C}}
\newcommand{\KK}{\mathbb{K}}
\newcommand{\HH}{\mathbb{H}}
\renewcommand{\AA}{\mathbb{A}}

\newcommand{\fm}{\mathfrak{m}}
\newcommand{\mn}{\mathfrak{n}}
\newcommand{\mo}{\mathfrak{o}}
\newcommand{\mpr}{\mathfrak{p}}
\newcommand{\mPr}{\mathfrak{P}}
\newcommand{\cA}{\mathcal{A}}
\newcommand{\cB}{\mathcal{B}}
\newcommand{\cC}{\mathcal{C}}
\newcommand{\cD}{\mathcal{D}}
\newcommand{\cF}{\mathcal{F}}
\newcommand{\cH}{\mathcal{H}}
\newcommand{\cK}{\mathcal{K}}
\newcommand{\cP}{\mathcal{P}}
\newcommand{\cR}{\mathcal{R}}
\newcommand{\cX}{\mathcal{X}}
\newcommand{\bG}{\mathbf{G}}
\newcommand{\bH}{\mathbf{H}}
\newcommand{\bO}{\mathbf{O}}
\newcommand{\bT}{\mathbf{T}}
\newcommand{\bchar}{\bm{1}}
\newcommand{\wG}{\widehat{G}_{\mathrm{sph}}}
\newcommand{\wmo}{\widehat{\mo}}
\newcommand{\eps}{\varepsilon}
\newcommand{\bs}{\backslash}
\newcommand{\GL}{\mathrm{GL}}
\newcommand{\PGL}{\mathrm{PGL}}
\newcommand{\SL}{\mathrm{SL}}
\newcommand{\PSL}{\mathrm{PSL}}
\newcommand{\SO}{\mathrm{SO}}
\newcommand{\SU}{\mathrm{SU}}
\newcommand{\M}{\mathrm{M}}
\newcommand{\Gr}{G^\mathrm{reg}}
\newcommand{\bGr}{\bG^\mathrm{reg}}
\newcommand{\Dr}{D^\mathrm{reg}}
\newcommand{\Kr}{K^\mathrm{reg}}
\newcommand{\diag}{\mathrm{diag}}
\newcommand{\ov}[1]{\overline{#1}}

\renewcommand{\geq}{\geqslant}
\renewcommand{\leq}{\leqslant}

\DeclareMathOperator{\vol}{vol}
\DeclareMathOperator{\Gal}{Gal}
\DeclareMathOperator{\Hom}{Hom}
\DeclareMathOperator{\Ind}{Ind}
\DeclareMathOperator{\Res}{Res}
\DeclareMathOperator{\re}{Re}

\DeclareMathOperator{\tr}{tr}
\DeclareMathOperator{\dist}{dist}
\DeclareMathOperator{\id}{id}
\DeclareMathOperator{\supp}{supp}
\DeclareMathOperator{\mm}{m}
\DeclareMathOperator{\nn}{n}
\DeclareMathOperator{\ram}{ram}
\DeclareMathOperator*{\wbigoplus}{\widehat{\bigoplus}}

\makeatletter
\DeclareRobustCommand\widecheck[1]{{\mathpalette\@widecheck{#1}}}
\def\@widecheck#1#2{%
    \setbox\z@\hbox{\m@th$#1#2$}%
    \setbox\tw@\hbox{\m@th$#1%
       \widehat{%
          \vrule\@width\z@\@height\ht\z@
          \vrule\@height\z@\@width\wd\z@}$}%
    \dp\tw@-\ht\z@
    \@tempdima\ht\z@ \advance\@tempdima2\ht\tw@ \divide\@tempdima\thr@@
    \setbox\tw@\hbox{%
       \raise\@tempdima\hbox{\scalebox{1}[-1]{\lower\@tempdima\box\tw@}}}%
    {\ooalign{\box\tw@ \cr \box\z@}}}
\makeatother

\theoremstyle{plain}
\newtheorem{theorem}{Theorem}
\newtheorem{proposition}{Proposition}
\newtheorem{lemma}{Lemma}

\theoremstyle{definition}
\newtheorem{definition}{Definition}
\newtheorem{remark}{Remark}

\newcommand\listsep{\setlength\itemsep{3pt}}
\numberwithin{equation}{section}

\begin{document}

\author{Miko\l aj Fr\c aczyk}
\author{Gergely Harcos}
\author{P\'eter Maga}
\author{Djordje Mili\'cevi\'c}

\address{Alfr\'ed R\'enyi Institute of Mathematics, Hungarian Academy of Sciences, POB 127, Budapest H-1364, Hungary}\email{gharcos@renyi.hu, magapeter@gmail.com}
\address{MTA R\'enyi Int\'ezet Lend\"ulet Automorphic Research Group}\email{gharcos@renyi.hu, magapeter@gmail.com}
\address{Institute for Advanced Study, Princeton NJ, USA}\email{mikolaj@ias.edu}
\address{Bryn Mawr College, Department of Mathematics, 101 North Merion Avenue, Bryn Mawr, PA 19010, USA}
\curraddr{Max-Planck-Institut f\"ur Mathematik, Vivatsgasse 7, D-53111 Bonn, Germany}
\email{dmilicevic@brynmawr.edu}

\title{The density hypothesis for horizontal families of lattices}

\begin{abstract}
We prove the density hypothesis for wide families of arithmetic orbifolds arising from all division quaternion algebras over all
number fields of bounded degree. Our power-saving bounds on the multiplicities of non-tempered representations
are uniform in the volume and spectral aspects.
\end{abstract}



\thanks{This research was supported by National Science Foundation Grant DMS-1903301 (D.M.), European Research Council grant CoG-648017 (M.F. \& G.H.), the MTA R\'enyi Int\'ezet Lend\"ulet Groups and Graphs Research Group (M.F.), the MTA R\'enyi Int\'ezet Lend\"ulet Automorphic Research Group (G.H. \& P.M.), NKFIH (National Research, Development and Innovation Office) grant K~119528 (G.H. \& P.M.), the Premium Postdoctoral Fellowship of the Hungarian Academy of Sciences (P.M.) and the National
Science Foundation Grant No. DMS-1638352 (M.F.)}

\maketitle

\section{Introduction}

\subsection{Exceptional spectrum}
Selberg's celebrated Eigenvalue Conjecture states that all nonzero Laplacian eigenvalues on congruence quotients of the upper half-plane are at least $1/4$. This is a strong form of the ``spectral gap'' property of the same eigenvalues being uniformly bounded away from $0$. In representation-theoretic language, Selberg's conjecture states that the archimedean constituents of non-trivial automorphic representations occurring in
the spherical discrete $L^2$-spectrum for congruence subgroups are tempered, that is, no complementary series representations occur. It is thus the archimedean counterpart of the Ramanujan--Petersson conjecture stating that the suitably normalized Hecke eigenvalues of cusp forms are bounded on primes, and is expected to hold for more general Lie groups and their arithmetic quotients\footnote{We remark that the Ramanujan--Petersson conjecture needs to be adjusted for general groups, as explained in \cite[\S 2]{Sar}.}.

Selberg's conjecture (as well as the Ramanujan conjecture for Maa{\ss} forms) remains far from being resolved. For individual forms the best known results are lower bounds on the spectral gap such as \cite{Kim2003}. For analytic applications in a \emph{family} of automorphic forms, in the absence of Selberg's conjecture, the non-tempered spectrum can often be satisfactorily handled if the exceptions in the family are known to be \emph{sparse and not too bad}. The situation is reminiscent of prime number theory, where a classical density theorem for Dirichlet $L$-functions $L(s,\chi)$
estimates the total number of zeros $\beta+i\gamma$ with $\beta\geq\sigma\geq 1/2$, $|\gamma|\leq T$, and $\chi$ modulo $q$ by $O_\eps((qT)^{c(1-\sigma)+\eps})$. Such a bound, conjectured with $c=2$ \cite[\S 10.1]{IwKo} and known with $c=12/5$ \cite[\S 10.4]{IwKo}, serves (in the light of explicit formulae) as a good substitute for the Riemann Hypothesis in applications such as counting primes in short intervals or arithmetic progressions.

It is imperative to understand which families admit an analogous density estimate for the exceptional spectrum, such as the ``correct'' power-saving bound (cf.\ \eqref{eq:DensityConjecture} and \eqref{p-lambda}) for the total multiplicity of Laplacian eigenvalues $1/4+(i\nu)^2$ with $\nu\geq\sigma\geq 0$ on congruence surfaces of increasing level or surfaces arising from Eichler orders in varying division quaternion algebras. Our main result, Theorem~\ref{thm:Main} below, covers these among many other examples and proves for the first time the density hypothesis for a natural broad ``horizontal'' family of not necessarily commensurable arithmetic orbifolds, including also the spectral aspect.

\subsection{Density theorems, uniformity, and implications}
We now state the density hypothesis for a family of lattices in a semisimple Lie group and discuss its applications and connections to the limit multiplicity problem. We refer to \cite{Blomer2019} for an overview of the density conjecture with a somewhat different focus.

\subsubsection{Density conjecture and density hypothesis}
Let $G$ be a semisimple real Lie group without compact factors, and let $k$ be a number field with ring of integers $\mo$. Let $\bG$ be a semisimple linear algebraic group over $k$ such that $\bG(k\otimes_\QQ\RR)$ is isomorphic to $G$ times a compact Lie group, and let $\iota$ denote the projection onto $G$. We identify $\bG(k)$ with its image under a faithful linear representation $\bG(k)\hookrightarrow\GL_n(k)$ and also with its image under the natural embedding $\bG(k)\hookrightarrow\bG(k\otimes_\QQ\RR)$. We define the \emph{principal congruence subgroups} as
\[\Gamma(\mn):=\iota\bigl(\{g\in\bG(k)\cap\GL_n(\mo) : g-\id_n\in\mn M_n(\mo)\}\bigr)\leq G,\]
where $\mn\subset\mo$ is a nonzero ideal. We refer to $k$ as the \emph{trace field} of $\Gamma(\mo)$.

Let $\pi\in\widehat{G}$ be an irreducible unitary representation of $G$. We assign it the real number
\begin{equation}\label{p-of-pi}
p(\pi):=\inf\big\{2\leq p\leq\infty : \text{$\pi$ has a nonzero matrix coefficient in $L^p(G)$}\big\}.
\end{equation}
A unitary representation $\pi$ is tempered if and only if $p(\pi)=2$. On the opposite side is the trivial representation $\bchar$ for which $p(\bchar)=\infty$. Using Harish-Chandra's expansion for matrix coefficients \cite[Ch.~VIII, \S 8]{Knapp} and the Howe--Moore vanishing theorem \cite[Th.~5.1]{HM}, it is easy to see that for $G$ simple the trivial representation is uniquely characterized by this property.

For every lattice $\Gamma\leq G$, let $\mm(\pi,\Gamma)$ be the multiplicity of a given $\pi\in\widehat{G}$ in the right regular representation of $G$ on $L^2(\Gamma\bs G)$. Sarnak and Xue conjectured \cite[Conj.~1]{SX} that
\begin{equation}\label{eq:PreDensity}\mm(\pi,\Gamma(\mn))\ll_\eps \vol(\Gamma(\mn)\bs G)^{2/p(\pi)+\eps}.
\end{equation}
This bound interpolates, for non-tempered representations, between the bound $\mm(\pi,\Gamma(\mn))\ll \vol(\Gamma(\mn)\bs G)$, which (at least for compact quotients) follows fairly directly from the trace formula and is sharp when $\pi\in\widehat{G}$ belongs to discrete series \cite{dGW78}, and the identity $\mm(\bchar,\Gamma(\mn))=1$. Sarnak and Xue \cite{SX} proved that \eqref{eq:PreDensity} holds for cocompact principal congruence lattices in $G=\SL_2(\RR)$ and $\SL_2(\CC)$.

Inequality \eqref{eq:PreDensity} admits a natural strengthening. For $\cB\subset\widehat{G}$, let
\begin{equation}\label{pdef}
p(\cB):=\inf\{p(\pi):\pi\in\cB\}.
\end{equation}
Then the \emph{density conjecture} (whose origins go back to \cite{Sar-ICM,SX}) states that, for every bounded subset $\cB\subset\widehat{G}$,
\begin{equation}\label{eq:DensityConjecture}\sum_{\pi\in\cB}\mm(\pi,\Gamma)\ll_{\cB,\eps}
\vol(\Gamma\bs G)^{2/p(\cB)+\eps}
\end{equation}
holds uniformly for all $\Gamma=\Gamma(\mn)$. It is known for $G=\SL_2(\RR)$ and $G=\SL_2(\CC)$ by \cite{HuKa93}, and also in further $\SL_2$ families of congruence quotients of a fixed lattice by \cite{BM,Hum,Hux,Iw2,Sar2} and others. Many partial results are known for other groups (see \cite{Blomer2019} as well as \cite{BBM,BrMi,Ja,MT}), although the full conjecture (including the individual bound \eqref{eq:PreDensity}) is still open if $G$ is any simple group of higher rank.

More generally, following \cite{KamGol19,KamGol20}, we say that a \emph{family} of lattices $\Gamma\leq G$ satisfies the \emph{density hypothesis} if \eqref{eq:DensityConjecture} holds uniformly for lattices in that family.

Recently, Golubev and Kamber~\cite{KamGol19,KamGol20} applied the Sarnak--Xue conjecture in various combinatorial contexts. Among their applications is the so-called \emph{optimal lifting property} \cite[Def.~1.4]{KamGol20}, which says that almost all elements of the quotient group $\Gamma(\mn)\bs \Gamma$ can be lifted to elements of $\Gamma$ lying in a ball of volume roughly $\vol(\Gamma(\mn)\bs G)$; clearly, no smaller ball suffices. By \cite[Thm.~1.6]{KamGol20}, which applies to more general families of lattices with a spectral gap, the optimal lifting property for the family $\Gamma(\mn)$ is implied by the following \emph{spherical density hypothesis}, a somewhat stronger version of \eqref{eq:DensityConjecture} that addresses the dependence of the implicit constant on $\cB$. Let $\wG\subset\widehat{G}$ be the spherical unitary dual of $G$, and let $\Omega(\pi)$ be the eigenvalue of the Casimir operator acting on $\pi\in\widehat{G}$. Following \cite[Def.~1.2]{KamGol20} we say that a family of lattices $\cF$ satisfies the spherical density hypothesis if there exists $L>0$ such that for every $\cB\subset\wG$ with $T:=\sup_{\pi\in\cB}\Omega(\pi)<\infty$ and every lattice $\Gamma\in\cF$ we have
\begin{equation}\label{eq:SphDensityHyp}
\sum_{\pi\in\cB}\mm(\pi,\Gamma)\ll_\eps \vol(\Gamma\bs G)^{2/p(\cB)+\eps}(1+T)^L.
\end{equation}
Our main result, Theorem~\ref{thm:Main} below, establishes this hypothesis for a wide family of lattices in $G=\SL_2(\RR)^a\times\SL_2(\CC)^b$.

\subsubsection{Limit multiplicity property}
Our extension of Sarnak--Xue's results was motivated by recent developments in the limit multiplicity property.
Let $\Gamma\leq G$ be a lattice. The right regular representation $L^2(\Gamma\bs G)$ decomposes into discrete part $L^2(\Gamma\bs G)_{\mathrm{disc}}$ and continuous part $L^2(\Gamma\bs G)_{\mathrm{cont}}$. We have
\[L^2(\Gamma\bs G)_{\mathrm{disc}}=\wbigoplus_{\pi\in\widehat{G}}\mm(\pi,\Gamma)\pi.\]
Let $\cF$ be a family of lattices in $G$ intersecting the center of $G$ in the same subgroup $\Theta$. We say that $\cF$ has the \emph{limit multiplicity property} if the measures
\[\mu_{\Gamma}:=\frac{1}{\vol(\Gamma\bs G)}\sum_{\pi\in\widehat{G/\Theta}}\mm(\pi,\Gamma)\delta_\pi,\qquad\Gamma\in\cF,\]
tend to the Plancherel measure $\mu_{\widehat{G/\Theta}}$ as $\vol(\Gamma\bs G)\to\infty$ in the following sense:
\begin{itemize}
\listsep
\item $\mu_\Gamma(\cB)\to\mu_{\widehat{G/\Theta}}(\cB)$ for every Jordan measurable subset $\cB\subset\widehat{G/\Theta}$ of tempered representations;
\item $\mu_\Gamma(\cB)\to 0$ for every bounded subset $\cB\subset\widehat{G/\Theta}$ of non-tempered representations.
\end{itemize}
The limit multiplicity property was discovered by DeGeorge and Wallach \cite{dGW78, dGW79}, and it implies a small-$o$ version of \eqref{eq:DensityConjecture}. That the distribution of $\pi\in\widehat{G/\Theta}$ occurring in $L^2(\Gamma\bs G)$ as $\vol(\Gamma\bs G)\to\infty$ should be guided by the Plancherel measure on $\widehat{G/\Theta}$ may also be thought of as the lattice (in particular, level) counterpart of the same property for the leading term in Weyl's law.

Until quite recently, the limit multiplicity property was only studied for families of principal congruence subgroups of a fixed arithmetic lattice. The state-of-the-art in this setting is due to Finis, Lapid and M\"uller~\cite{FiLaMu15, FiLa18}. The situation changed when Ab\'ert et\ al. proved \cite[Th.~1.2]{7Sam} that if $G$ is a simple group of higher rank, then the limit multiplicity property holds for any family of lattices $\cF$ satisfying the following condition:
\begin{center}
there exists an open set $U\subset G$ such that $g\Gamma g^{-1}\cap U=\{1\}$ for every $\Gamma\in\cF$ and $g\in G$.
\end{center}
This is automatically satisfied if $\cF$ consists of cocompact torsion-free arithmetic lattices whose trace field is of bounded degree over $\QQ$ (see \cite[\S 10]{Gelander} and \cite[Th.~1]{Dobro}). Using other methods, the first named author proved in \cite{Fra2016} that the family of all cocompact torsion-free congruence lattices of $\SL_2(\RR)$ and $\SL_2(\CC)$ has the limit multiplicity property.

\subsection{Main result}
Our main result, Theorem~\ref{thm:Main} below, establishes a strong form of the density hypothesis for natural broad ``horizontal'' families of arithmetic lattices, arising from all suitable division quaternion algebras over varying number fields. We describe these families of lattices in \S\ref{lattices-subsubsection} and \S\ref{congruence-subsubsection}, the spherical unitary representations entering \eqref{eq:DensityConjecture} in \S\ref{reps-subsubsection}, and then state Theorem~\ref{thm:Main} in \S\ref{main-thm-subsubsection}.

\subsubsection{Arithmetic lattices and orbifolds}\label{lattices-subsubsection}
We begin with a standard construction of arithmetic lattices using quaternion algebras.

Let $a,b,c\in\NN$ with $a+b\geq 1$. Let $k$ be a number field of degree $d=a+2b+c$ with $b$ complex places. We enumerate the archimedean completions of $k$ by $k_1,\dotsc,k_{a+b+c}$, with the complex ones being $k_{a+1},\dotsc,k_{a+b}$. Let $A$ be a quaternion algebra over $k$ of signature $(a,b,c)$, so that $A\otimes_\QQ\RR\simeq\M_2(\RR)^a\times\M_2(\CC)^b\times\HH^c$,
where $\HH=\bigl(\frac{-1,-1}{\RR}\bigr)$ stands for the Hamilton quaternion algebra. Let $\nn$ (resp. $\tr$) be the reduced norm (resp. reduced trace) on $A$. We will write $\ram(A)$ for the set of places of $k$ where $A$ ramifies.

Let $\bG:=\SL_1(A)$ viewed as an algebraic group over $k$, and let $G:=\SL_2(\RR)^a\times\SL_2(\CC)^b$. For every $k$-algebra $R$ we have
\[\bG(R)=\{x\in A\otimes_k R : \nn(x)=1\}.\]
In particular, writing $\AA=\AA_\infty\times\AA_f$ for the adele ring of $k$, we have
\[\bG(\AA_{\infty})\simeq\bG(k\otimes_\QQ\RR)\simeq\{x\in A\otimes_\QQ\RR : \nn(x)=1\}\simeq G\times\SU_2(\CC)^c,\]
and hence
\begin{equation}\label{GA-decomp}
\bG(\AA)\simeq G\times\SU_2(\CC)^c\times\bG(\AA_f).
\end{equation}
Write $\eta\colon\bG(\AA)\to G$ for the projection onto the first factor in \eqref{GA-decomp}. It is determined uniquely up to
automorphisms of $G$. For a compact open subgroup $U\leq\bG(\AA_f)$, the group
\[\Gamma_U:=\eta\bigl(\bG(k)\cap(\bG(\AA_\infty)\times U)\bigr)\leq G\]
is a congruence lattice of $G$. Let $K$ be a maximal compact subgroup of $G$. Then, $G/K\simeq (\cH^2)^a\times(\cH^3)^b$, and its quotient $\Gamma_U\bs G/K$ is an arithmetic orbifold, which is compact if and only if $A$ is a division quaternion algebra (as will be the case in our Theorem~\ref{thm:Main}). Its commensurability class is uniquely determined by the pair $(k,A)$, which are referred to as its invariant trace field and invariant trace quaternion algebra \cite[\S\S 4--5]{B}. This description includes, for $G=\SL_2(\RR)$ and $G=\SL_2(\CC)$, all arithmetic hyperbolic $2$- and $3$-orbifolds.

\subsubsection{Families of congruence lattices}\label{congruence-subsubsection}
For a nonzero ideal $\mn\subset\mo$ not divisible by any $\mpr\in\ram(A)$, we introduce certain natural congruence lattices $\Gamma_{\kappa}(\mn)\leq G$ closely related to the classical congruence subgroups $\Gamma_0(n),\Gamma(n)\leq\SL_2(\ZZ)$. Fix an isomorphism
\begin{equation}\label{eq:Ident1}
\bG(\AA_f)\simeq\prod_{\mpr\in\ram(A)}\SL_1(D_\mpr)\times \prod_{\mpr\not\in\ram(A)}\SL_2(k_{\mpr}),
\end{equation}
where $D_{\mpr}$ is the unique non-split quaternion algebra over $k_{\mpr}$. The group $\SL_1(D_\mpr)$ is anisotropic over $k_\mpr$ (i.e., it has no split tori), so it is compact by \cite[\S 9.4]{BT65}; see also \cite[\S 6.4]{B}. Writing $\mn_\mpr:=\mn\mo_\mpr$, consider the local compact open subgroups $K_0(\mn_\mpr),K_1(\mn_\mpr)\leq\SL_2(\mo_\mpr)$ defined as in \eqref{K0}--\eqref{K1} by
\[K_0(\mn_\mpr):=\left\{\begin{pmatrix}a & b \\ c & d\end{pmatrix}\in\SL_2(\mo_\mpr): c\in\mn_\mpr\right\},\qquad
 K_1(\mn_\mpr):=\left\{\begin{pmatrix}a & b \\ c & d\end{pmatrix}\in\SL_2(\mo_\mpr): a-d,b,c\in\mn_\mpr\right\}\]
For any $\kappa\colon\{\mpr\mid\mn\}\to\{0,1\}$, define the compact open subgroup $K_{\kappa}(\mn)\subset\bG(\AA_f)$ as
\begin{equation}\label{K-def}
K_{\kappa}(\mn):=\prod_{\mpr\in\ram(A)}\SL_1(D_\mpr)\times\prod_{\mpr\mid\mn}K_{\kappa(\mpr)}(\mn_\mpr)
\,\times\prod_{\substack{\mpr\not\in\ram(A)\\\mpr\nmid\mn}}\SL_2(\mo_\mpr),
\end{equation}
and the corresponding congruence lattice $\Gamma_{\kappa}(\mn)\leq G$ as
\begin{equation}\label{Gamma-def}
\Gamma_{\kappa}(\mn):=\Gamma_{K_{\kappa}(\mn)}.
\end{equation}
We remark that $\Gamma_{\kappa}(\mn)$ depends on the choice of the identification \eqref{eq:Ident1}, but our estimates will not depend on this choice. We are now ready to declare the family of lattices which Theorem~\ref{thm:Main} is concerned with.
\begin{definition}\label{lattices-def}
For $a,b,c\in\NN$, let $\cF_{a,b,c}$ be the family of all congruence lattices $\Gamma_{\kappa}(\mn)\leq\SL_2(\RR)^a\times\SL_2(\CC)^b$ as in \eqref{Gamma-def}, where $k$ is a number field of degree $a+2b+c$, $A$ is a division quaternion algebra over $k$ of signature $(a,b,c)$, $\mn\subset\mo_k$ is a nonzero ideal, and $\kappa\colon\{\mpr\mid\mn\}\to\{0,1\}$.
\end{definition}
Thus $\cF_{a,b,c}$ is a family of cocompact congruence lattices in $G$, which consists of many different commensurability classes according to the pair $(k,A)$.

\subsubsection{Representations}\label{reps-subsubsection}
Let $S_{\infty}^G=\{1,\dotsc,a+b\}$. The spherical unitary dual $\wG$ can be parametrized as $\pi_{\bm{s}}$ with tuples $\bm{s}=(s_j)$ satisfying $s_j\in(0,1/2]\cup i[0,\infty)$ for all $j\in S_{\infty}^G$ as follows. For $s\in\CC$ and $\KK\in\{\RR,\CC\}$, let $\pi^\KK_s$ be the spherical principal series representation of $\SL_2(\KK)$ of normalized Casimir eigenvalue $1/4-s^2$. In particular, $\pi^\KK_{-s}=\pi^\KK_s$. The point $s=1/2$ corresponds to the trivial representation, the interval $(0,1/2)$ corresponds to the non-tempered spherical unitary spectrum, and the half-line $i[0,\infty)$ corresponds to the tempered spherical unitary spectrum. We give more background on spherical representations in \S\ref{sec:Spherical}. For $\bm{s}=(s_j)_{j=1}^{a+b}$, we define $\pi_{\bm{s}}:=\bigotimes_{j=1}^{a+b}\pi^{k_j}_{s_j}$. It lies in $\widehat{G}$ if and only if each $s_j$ lies in $[-1/2,1/2]\cup i\RR$, and in $\widehat{G}_\mathrm{temp}$ if and only if each $s_j$ lies in $i\RR$. Using \cite[Th.~8.48]{Knapp}, we may verify that (cf.\ \eqref{p-of-pi})
\begin{equation}\label{p-of-pi-s}
p(\pi_{\bm{s}})=\max_{j\in S_{\infty}^G}p(\pi_{s_j}^{k_j})=2/\min_{j\in S_{\infty}^G}(1-2|\!\re(s_j)|),\qquad\pi_{\bm{s}}\in\widehat{G}.
\end{equation}

For $S\subset S_{\infty}^G$, $\bm{\sigma}=(\sigma_j)\in[0,1/2]^S$, and $\bm{T}=(T_j)\in\RR^{S_{\infty}^G\setminus S}$, we introduce the following bounded subset of $\wG$:
\begin{equation}\label{set-lambda-T}
\cB(\bm{\sigma},\bm{T}):=\biggl\{\pi_{\bm{s}} \ : \ \bm{s}\in\prod_{j\in S}[\sigma_j,1/2]\,\times
\prod_{j\in S_{\infty}^G\setminus S}i[T_j-1,T_j+1]\biggr\}
\end{equation}
It is clear from \eqref{pdef} and \eqref{p-of-pi-s} that
\begin{equation}\label{p-lambda}
p(\cB(\bm{\sigma},\bm{T}))=p(\bm{\sigma}):=\begin{cases}
2/\min_{j\in S}(1-2\sigma_j),&\text{if $S\neq\emptyset$,}\\
2,&\text{if $S=\emptyset$.}\end{cases}
\end{equation}

\subsubsection{Statement of the main result}\label{main-thm-subsubsection}
Let $\rho_j:=1$ for $j\in\{1,\dotsc,a\}$ and $\rho_j:=2$ for $j\in\{a+1,\dotsc,a+b\}$. In the spirit of the analytic conductor of Iwaniec--Sarnak~\cite{IwSa},
we define for an arbitrary lattice $\Gamma\in\cF_{a,b,c}$ and tuple $\bm{T}\in\RR^{S_{\infty}^G\setminus S}$ the quantity
\begin{equation}\label{conductors}
\cC(\Gamma,\bm{T}):=\vol(\Gamma\bs G)\prod_{j\in S_{\infty}^G\setminus S}(1+|T_j|)^{\rho_j}.
\end{equation}
We note that $\vol(\Gamma\bs G)$ always exceeds $e^{-7}$ by Borel's volume formula and Odlyzko's discriminant bound (see Proposition~\ref{prop:Borel}). Recalling Definition~\ref{lattices-def} and \eqref{set-lambda-T}--\eqref{conductors}, our main result reads as follows.
\begin{theorem}\label{thm:Main}
For every $a,b,c\in\NN$, the family $\cF_{a,b,c}$ of congruence lattices in $G=\SL_2(\RR)^a\times\SL_2(\CC)^b$ satisfies the spherical density hypothesis. More precisely, for every $\Gamma\in\cF_{a,b,c}$, $S\subset S_{\infty}^G$, $\bm{\sigma}\in[0,1/2]^S$, and $\bm{T}\in\RR^{S_{\infty}^G\setminus S}$, we have for any $\eps>0$
\begin{equation}\label{main-bound}
\sum_{\pi\in\cB(\bm{\sigma},\bm{T})}\mm(\pi,\Gamma)\ll_{\eps,a,b,c}
\cC(\Gamma,\bm{T})^{2/p(\bm{\sigma})+\eps}.
\end{equation}
\end{theorem}

Already for $S=S_{\infty}^G$ (or for fixed $\bm{T}\in\RR^{S_{\infty}^G\setminus S}$), Theorem~\ref{thm:Main} extends for the first time the results of Sarnak--Xue \cite{SX} to families of non-commensurable lattices of $G$. Theorem~\ref{thm:Main} allows for groups $G$ of arbitrary rank $a+b$ and a wider variety of congruence subgroups $\Gamma_{\kappa}(\mn)$, even when considering subgroups of a fixed lattice. Moreover, it holds uniformly over all lattices in $\cF_{a,b,c}$ and all possible pairs $(k,A)$, with no dependence on a particular fixed ambient lattice, and addresses for the first time the dependence on the tempered components of $\pi$. In the fully degenerate case $S=\emptyset$, Theorem~\ref{thm:Main} recovers up to $\vol(\Gamma\bs G)^\eps$ the bound resulting from Weyl's law for the group $G$ (cf.\ \eqref{slightly-stronger} and \cite[Prop.~7.3]{BrMi}), for the first time uniformly across all lattices in $\cF_{a,b,c}$. Finally, \eqref{main-bound} is in fact significantly stronger than \eqref{eq:SphDensityHyp} specialized to $\cF_{a,b,c}$, in that it exhibits a natural dependence of the archimedean parameters (cf.\ last paragraph of \S\ref{new-features-sec}).

The dependence of the implicit constant in Theorem~\ref{thm:Main} on the signature $(a,b,c)$, or equivalently on the degree $[k:\QQ]$, seems difficult to remove. In particular, parts of our argument that appeal to the geometry of ideal lattices of $k$ are very sensitive on the degree. This contrasts with the limit multiplicity property for all cocompact torsion-free congruence lattices of $\SL_2(\RR)$ and $\SL_2(\CC)$, obtained in \cite{Fra2016} without any restriction on $[k:\QQ]$.

\subsection{New features}\label{new-features-sec}
Theorem~\ref{thm:Main} will be proved by the Arthur--Selberg trace formula, with a test function built from a positive definite $f\in C_c(G)$, chosen so as to emphasize contributions of $\pi$ occurring on the left hand side of \eqref{main-bound}. We give a detailed overview of the method in \S\ref{sec:method}, contending ourselves here with just pointing out several new features. The first is the \emph{self-normalizing} nature of our estimates. Indeed, after the application of the trace formula with our specific test function, we may bound the left hand side of \eqref{main-bound} by (see \S\ref{notations} for notations, \S\ref{sec:method} and \S\ref{sec:Proof} for more details)
\begin{equation}\label{post-TF-intro}
e^{(2/p(\bm{\sigma})-1)R}\sum_{[\gamma]\subset\bG(k)}\vol(\bG_{\gamma}(k)\bs\bG_{\gamma}(\AA))\cdot\bO(\gamma,f)
\bO(\gamma,\bchar_{\SU_2(\CC)^c})\bO(\gamma,\bchar_{K_{\kappa}(\mn)}),
\end{equation}
In this expression, we estimate the volumes building on the work of Borel~\cite{B}, Ono~\cite{Ono,Ono2}, and Ullmo--Yafaev~\cite{UY} (see Propositions~\ref{prop:Borel}--\ref{prop:TorusVolume}); the orbital integrals using various integral
transforms and counting in Bruhat--Tits trees (see Propositions~\ref{nontempered}--\ref{proposition:orbital_integral_in_nonarch}); the number of contributing conjugacy classes using the geometry of numbers. And yet, the resulting estimates fit essentially perfectly, estimating the contribution of regular conjugacy classes to \eqref{post-TF-intro} as
\begin{equation}\label{self-normalized}
\preccurlyeq e^{(2/p(\bm{\sigma})-1)R}\cdot e^R\Delta_k^{-1/2}\cdot \Delta_k^{1/2}|N_{k/\QQ}(\Delta_{k(\gamma)/k})|^{1/2}\cdot |N_{k/\QQ}(\Delta(\gamma))|^{-1/2}\cdot\frac{|N_{k/\QQ}(\Delta(\gamma))|^{1/2}}{|N_{k/\QQ}(\Delta_{k(\gamma)/k})|^{1/2}}w_{\kappa}^{\mn}(\tr\gamma).
\end{equation}
Here $w_{\kappa}^{\mn}(\tr\gamma)$ is an explicit factor, which may be large but is $\preccurlyeq 1$ in a certain average sense (suitable for us), and $X\preccurlyeq Y$ stands for $X\ll_\eps\cC(\Gamma,\bm{T})^\eps Y$. Following remarkable cancellations in
\eqref{self-normalized}, we obtain Theorem~\ref{thm:Main} by making the choice $R:=7+\log\cC(\Gamma,\bm{T})$, which matches the central and regular contributions.

In particular, when \emph{counting conjugacy classes} contributing to \eqref{post-TF-intro}, we first prove (see Lemma~\ref{lem:ConjCount}) that the classes in $\bGr(k)$ intersecting the union of all $\bG(\AA_f)$-conjugates of $K_\emptyset(\mo)$ are determined up to $\preccurlyeq 1$ possibilities by their traces, a new feature for $\bG=\SL_1(A)$. This reduces the problem to bounding the sum of $w_{\kappa}^{\mn}(x)$ over $x\in\mo$ (which is a weighted count of points in cosets of ideal lattices) lying in a specified, typically highly unbalanced polycylinder in $\AA_\infty$. For this purpose, we have developed Lemma~\ref{lem:CylinderSums}, a variant of a geometry of numbers result from \cite{FHM20}, capturing that ideal lattices in $\AA_\infty$ are not too skew, even with a varying $k$; see \S\ref{counting-classes-overview-subsec}.

The emergence of $\cC(\Gamma,\bm{T})$ in Theorem~\ref{thm:Main}, as the natural guiding parameter for the density hypothesis, appears to be novel. It is essentially the \emph{analytic conductor}, encapsulating the complexity at all places of $k$. Its role is best understood by rewriting \eqref{main-bound} in the form in which it arises in our arguments, namely as
\[ \sum_{\pi\in\cB(\bm{\sigma},\bm{T})}\mm(\pi,\Gamma)\cdot e^{(1-2/p(\bm{\sigma}))R}\preccurlyeq\cC(\Gamma,\bm{T}), \]
where $R:=7+\log\cC(\Gamma,\bm{T})$ is related to the allowable support of $f$ in \eqref{post-TF-intro}. The right hand side, which includes a majorizer for the Plancherel volume of the unit ball around any $\pi\in\cB(\bm{\sigma},\bm{T})$ and thus arises naturally from the central terms in the trace formula detecting such $\pi$ in $L^2(\Gamma\bs G)$, is a famous barrier in the multiplicity problem that has never been overcome on a power scale. Our ability to do so for $p(\bm{\sigma})>2$, and in fact uniformly in both factors in \eqref{conductors}, is due to our ability to increase $R$ symmetrically in the self-normalized bound \eqref{self-normalized}. See also \cite{Hum} for estimates with power savings at all places of $k=\QQ$ on congruence surfaces.

As in \cite[Prop.~7.3]{BrMi}, our estimates on the geometric terms (Propositions~\ref{proposition:orbital_integral_in_nonarch}--\ref{prop:TorusVolume}) along with Theorem~\ref{thm:Main} and some additional care with
archimedean test functions (extending Proposition~\ref{tempered}) should allow for a uniform sharp-cutoff Weyl's law for $\mm(\pi_{i\bm{T}},\Gamma)$ with $\bm{T}$ in rather general bounded regions in $\mathbb{R}^{S_{\infty}^G}$, uniformly over $\Gamma\in\cF_{a,b,c}$ and with a $1/\log\cC(\Gamma,{\bm T})$-savings in the error term; cf.~\cite[\S 3.1]{BrMi}.

\subsection{Outline of the paper}
In \S\ref{sec:method}, we give a detailed overview of the method, worked out in a specific instance chosen to illustrate all key steps.
In \S\S\ref{sec:Arch}--\ref{sec:Global}, we prove all essential ingredients to estimate contributions of non-central conjugacy classes to the trace formula. In \S\ref{sec:Arch}, we fix the Haar measures at the archimedean places, construct the archimedean test functions, and estimate their orbital integrals. In \S\ref{sec:NonArch}, we fix the Haar measures at the non-archimedean places, and estimate the orbital integrals of characteristic functions of the relevant congruence subgroups of $\SL_2(\mo_\mpr)$. In \S\ref{sec:Borel}, we adapt Borel's volume formula to our situation. In \S\ref{sec:Volumes}, we estimate the volume of $\bG_\gamma(k)\bs\bG_\gamma(\AA)$ for regular semisimple elements $\gamma\in\bG(k)$. In \S\ref{sec:ConjClasses}, we bound the number of rational conjugacy classes $[\gamma]\subset\bG(k)$ of a fixed trace that bring nonzero contributions to the trace formula. Finally, in \S\ref{sec:Proof} we set up the global test function and combine all the ingredients to prove Theorem~\ref{thm:Main}.

\subsection{Notations}\label{notations}
We denote by $\NN:=\{0,1,2,\dotsc\}$ the set of natural numbers. Let $a,b,c\in\NN$ as before. Throughout this article, $k$ is a number field of degree $a+2b+c$, with ring of integers $\mo$ and ring of adeles $\AA=\AA_\infty\times\AA_f$. We enumerate the archimedean places of $k$ by $j\in\{1,\dotsc,a+b+c\}$ in such a way that $j\in\{a+1,\dotsc,a+b\}$ correspond to all complex places of $k$. We write $\rho_j:=[k_j:\RR]$. For any nonzero prime ideal $\mpr\subset\mo$, we write $k_\mpr$ for the $\mpr$-adic completion of $k$, $\mo_\mpr$ for the ring of integers of $k_\mpr$, and $q$ for the size of residue field $\mo/\mpr$. We also write $\wmo:=\prod_\mpr\mo_\mpr$. We shall use the Haar probability measure on $\wmo$ and on the factors $\mo_\mpr$. We write $\Delta_k$ and $N=N_{k/\QQ}$ for the absolute discriminant and norm in $k/\QQ$, and we also write $\Delta_{l/k}$ and $N_{l/k}$ for the relative discriminant and norm in an extension $l/k$ (both valued in ideals in $\mo$), with analogous notations for local fields.

We denote by $A$ a division quaternion algebra over $k$ of signature $(a,b,c)$, by $\bG$ the algebraic group $\SL_1(A)$ defined over $k$, and by $\ram(A)$ the set of places ramified in $A$. If $\gamma\in A$, we write $k(\gamma)$ for the subfield generated by $\gamma$, which is quadratic unless $\gamma$ is in the center of $A$. We write $[\gamma]$ for the conjugacy class of $\gamma$, in the group $\bG(k)$ or in one of its discrete subgroups $\Gamma$, depending on the context. We denote by $\bGr(k)$ the set of regular semisimple elements of $\bG(k)$, and by $\bG_\gamma$ the centralizer of $\gamma$ in $\bG$, with analogous notations $\Gr$ and $G_{\gamma}$ in the group $G$ as in the next paragraph.

In \S\ref{sec:Arch}, $G$ is either $\SL_2(\RR)$ or $\SL_2(\CC)$. In \S\ref{sec:NonArch}, $G$ stands for $\SL_2(F)$ for a non-archimedean local field $F$. Outside of these two sections, we will write $G=\SL_2(\RR)^a\times \SL_2(\CC)^b$. We write $\widehat{G}$ for the unitary dual of $G$ and $\wG$ for the subset of spherical representations.

For $x\in\RR$, we write $x^+:=\max(x,0)$. The notation $X\ll_{d,e,\dotsc}Y$ means that there exists a constant $C=C(a,b,c,d,e,\dotsc)$ such that $|X|\leq C Y$. In particular, all implicit constants are allowed to depend on the signature $(a,b,c)$ or equivalently on the degree $[k:\QQ]$, but are otherwise absolute except as indicated by a subscript. We also write $X\asymp Y$ to indicate that $X\ll Y\ll X$. Finally, by $X\preccurlyeq Y$ we mean that $X\ll_\eps\cC(\Gamma,\bm{T})^\eps Y$ holds for all $\eps>0$.

\subsection{Acknowledgements}
We are truly grateful to Gerg\H o Nemes for proving various bounds involving the Legendre function. These bounds were not necessary in the end, but they played an important role at an earlier stage of the manuscript. We thank Tobias Finis for useful discussions about the limit multiplicity property. Finally, Djordje Mili\'cevi\'c would like to thank the Max Planck Institute for Mathematics for their support and exceptional research infrastructure.

\section{Method}\label{sec:method}
In this section, we give a detailed outline of the proof, using a specific example to simplify matters while capturing the most important features. In this spirit, we consider the case when $G=\SL_2(\RR)^2$, $k$ is a totally real field of degree $4$, $A$ is the unique quaternion algebra over $k$ ramified exactly at the third and fourth archimedean place (i.e., $a=2$, $b=0$, $c=2$), and $\Gamma=\Gamma_\emptyset(\mo)$. We settle for bounding the multiplicity in $L^2(\Gamma\bs G)$ of a single non-tempered spherical representation, $\pi_{\sigma,iT}$ for $\sigma\in(0,1/2)$ and $T\in\RR$, as in \eqref{eq:PreDensity}.

Borel's volume formula (Proposition~\ref{prop:Borel}) yields $\vol(\Gamma\bs G)\asymp\Delta_k^{3/2}$, so we need to prove
\begin{equation}\label{needtoprove}
\mm(\pi_{\sigma,iT},\Gamma)\ll_\eps\bigl(\Delta_k^{3/2}(1+|T|)\bigr)^{1-2\sigma+\eps}.
\end{equation}

\subsection{Trace formula setup}
Following the approach of \cite{dGW78,SX} we estimate the multiplicity by the trace of a suitably chosen positive definite operator acting on $L^2(\Gamma\bs G)$. Let $f\in C_c(G)$, and let $Rf\colon L^2(\Gamma\bs G)\to L^2(\Gamma\bs G)$ be the operator
\begin{equation}\label{R-gamma-def}
(Rf \phi)(h):=\int_{G}\phi(hg)f(g)\,dg.
\end{equation}
The quotient $\Gamma\bs G$ is compact, so $Rf$ is of trace class, and we have
\begin{equation}\label{trace-formula-outline}
\tr Rf=\sum_{\pi\in\widehat{G}}\mm(\pi,\Gamma)\tr\pi(f)= \sum_{[\gamma]\subset \Gamma}\vol(\Gamma_\gamma\bs G_\gamma)\bO(\gamma,f),
\end{equation}
where $[\gamma]$ runs through the set of conjugacy classes of $\Gamma$, and
\[\bO(\gamma, f):=\int_{G_\gamma\bs G}f(g^{-1}\gamma g)\,d\dot g.\]
We shall choose a test function of the form $f=f_1\otimes f_2$, where $f_1,f_2\in C_c(\SL_2(\RR))$ are positive definite. We want to find $f$ such that $\tr\pi_{\sigma,iT}(f)$ is big, while the right hand (geometric) side of the trace formula remains relatively small. Since $f\in C_c(G)$ is positive definite, this will automatically lead to an upper bound on $\mm(\pi_{\sigma,iT},\Gamma)$.

We switch to the adelic version of the trace formula in order to make the orbital integrals, the volumes, and the set of conjugacy classes more manageable. Let $f_\AA:=f \otimes \bchar_{\SU_2(\CC)}^{2}\otimes \bchar_{K_{\emptyset}(\mo)}$. In Section~\ref{sec:Global}, we normalize the Haar measure on $\bG(\AA)$ such that $\vol(\bG(k)\bs \bG(\AA))=\vol(\Gamma\bs G)$. By \eqref{eq:geom_exp} we have
\begin{equation}\label{eq:TraceForm1}
\tr Rf=\tr Rf_\AA=\sum_{[\gamma]\subset \bG(k)} \vol(\bG_\gamma(k)\bs\bG_\gamma(\AA))\bO(\gamma, f_\AA),
\end{equation}
where the sum is now taken over the conjugacy classes of $\bG(k)$.

Since the test function $f_\AA$ is given as a pure tensor, the orbital integrals in \eqref{eq:TraceForm1} decompose as
\begin{equation}\label{methoddecomp}
\bO(\gamma,f_\AA)=\bO(\gamma,f_1)\bO(\gamma,f_2) \prod_{v\in\ram_{\infty}(A)}\bO(\gamma_v,\bchar_{\SU_2(\CC)})\prod_{\mpr}\bO(\gamma,\bchar_{\SL_2(\mo_{\mpr})}).
\end{equation}
The bounds for the local orbital integrals will naturally involve the Weyl discriminant, which is an element of $k$ defined for $\gamma\in A(k)\subset\SL_2(\overline{k})$ with eigenvalues $w,w^{-1}\in\ov{k}^\times$ as\footnote{This differs from the general definition of Weyl discriminant $D(\gamma):=\det(1-{\rm Ad}(\gamma))|_{\mathfrak{g}/\mathfrak{g}_\gamma}$ by the minus sign.}
\[\Delta(\gamma):=(w-w^{-1})^2=\tr(\gamma)^2-4.\]

\subsection{Archimedean test functions}
Let $R\geq 0$ be a parameter. Using Harish-Chandra's spherical transform, we construct in Propositions~\ref{nontempered} and \ref{tempered} spherical test functions $f_1,f_2\in C_c(\SL_2(\RR))$ with the following properties:
\begin{itemize}
\listsep
\item Large trace: $\tr\pi^{\RR}_{\sigma}(f_1)\gg e^{2\sigma R}$ and $\tr\pi^{\RR}_{iT}(f_2)\gg 1$.
\item Controlled support: $\supp f_1\subset B(2R+2)$ and $\supp f_2\subset B(2)$, where $B(R)$ is as in \eqref{BRdef}.
\item Small orbital integrals: $\bO(\gamma, f_j)\ll |\Delta(\gamma)|_j^{-1/2}$ for any $\gamma\in\bGr(k)$.
\item Boundedness: $f_1(g)\ll 1$ and $f_2(g)\ll 1+|T|$.
\end{itemize}
These properties imply that
\begin{itemize}
\listsep
\item $\tr\pi_{\sigma,iT}(f)\gg e^{2\sigma R}$;
\item the sum in \eqref{eq:TraceForm1} is taken over $[\pm\id]$ and the set $W(R)$ of $[\gamma]\subset\bGr(k)$ such that the conjugacy class of $\gamma$ in $\bG(\AA)$ intersects $B(2R+2)\times B(2)\times\SU_2(\CC)^{2}\times\SL_2(\wmo)$;
\item $\bO(\gamma, f_1)\bO(\gamma,f_2)\ll |N_{k/\QQ}(\Delta(\gamma))|^{-1/2}$;
\item $f(\pm\id)\ll 1+|T|$.
\end{itemize}

The test function $f=f_1\otimes f_2\in C_c(G)$ is positive definite, so from \eqref{trace-formula-outline}, \eqref{eq:TraceForm1} and \eqref{methoddecomp} we get
\begin{equation}\label{eq:TraceFormEst1}
\mm(\pi_{\sigma,iT})\ll e^{-2R\sigma}\bigg(\vol(\bG(k)\bs\bG(\AA))(1+|T|)+\sum_{[\gamma]\in W(R)}
\frac{\vol(\bG_\gamma(k)\bs\bG_\gamma(\AA))\bO(\gamma,\bchar_{\SL_2(\wmo)})}{|N_{k/\QQ}(\Delta(\gamma))|^{1/2}}\bigg).
\end{equation}
The first summand comes form the central conjugacy classes $[\gamma]=[\pm\id]$. Our strategy is to choose the maximal $R\geq 0$ for which the central terms still dominate the sum (up to a negligible factor). We will see later that the correct choice is provided by
\begin{equation}\label{correctchoice}
e^R=e^7\cdot\cC(\Gamma,\bm T)\asymp\vol(\bG(k)\bs\bG(\AA))(1+|T|)\asymp\Delta_k^{3/2}(1+|T|).
\end{equation}
The factor $e^7$ ensures that $R\geq 0$ (see Proposition~\ref{prop:Borel}).

\subsection{Estimating the geometric terms}
The non-archimedean orbital integrals are estimated in Section~\ref{sec:NonArch} using Bruhat--Tits trees. This treatment allows us to achieve great uniformity in the number field $k$, ramification of the quaternion algebra $A$, and the compact open subgroup of $\bG(\AA_f)$ giving rise to the lattice $\Gamma$. In the present case, our Proposition~\ref{proposition:orbital_integral_in_nonarch} shows that
\[\left|\bO(\gamma,\bchar_{\SL_2(\mo_{\mpr})})\right|\leq\begin{cases}1,&\text{if $|\Delta(\gamma)|_{\mpr}=1$},\\[2pt]
C{|\Delta(\gamma)|}_{\mpr}^{-1/2}{|\Delta_{k(\gamma)/k}|}_{\mpr}^{1/2},&\text{otherwise},\end{cases}\]
where $C>0$ is an absolute constant. Multiplying this bound over all non-archimedean places $\mpr$, and denoting by $\omega(\Delta(\gamma))$ the number of distinct prime ideals diving $\Delta(\gamma)$, we get
\[\left|\bO(\gamma,\bchar_{\SL_2(\wmo)})\right|\leq C^{\omega(\Delta(\gamma))}\prod_\mpr{|\Delta(\gamma)|}_{\mpr}^{-1/2}{|\Delta_{k(\gamma)/k}|}_{\mpr}^{1/2}
\ll_\eps e^{\eps R}\frac{|N_{k/\QQ}(\Delta(\gamma))|^{1/2}}{|N_{k/\QQ}(\Delta_{k(\gamma)/k})|^{1/2}},\]
since $[\gamma]\in W(R)$ forces an upper bound $|N_{k/\QQ}(\Delta(\gamma))|\ll e^{2R}$ (cf.\ \eqref{where-trace-is}), so that $C^{\omega(\Delta(\gamma))}\ll_{\eps} e^{\eps R}$.

Volumes of adelic quotients of algebraic tori such as $\bG_{\gamma}$ in \eqref{eq:TraceFormEst1} have been computed in \cite{UY}. We show in Proposition~\ref{prop:TorusVolume} that
\[\vol(\bG_\gamma(k)\bs \bG_\gamma(\AA))\ll_\eps\Delta_k^{1/2+\eps}|N_{k/\QQ}(\Delta_{k(\gamma)/k})|^{1/2+\eps}.\]
Combining the above bounds on volumes of adelic quotients and non-archimedean orbital integrals, the contribution of a single regular conjugacy class $[\gamma]\in W(R)$ to \eqref{eq:TraceFormEst1} satisfies
\[\frac{\vol(\bG_\gamma(k)\bs\bG_\gamma(\AA))\bO(\gamma,\bchar_{\SL_2(\wmo)})}{|N_{k/\QQ}(\Delta(\gamma))|^{1/2}}\ll_\eps e^{4\eps R}\Delta_k^{1/2}.\]

Returning to \eqref{eq:TraceFormEst1}, and using also \eqref{correctchoice}, we conclude that
\[\mm(\pi_{\sigma,iT},\Gamma)\ll_\eps e^{-2R\sigma+4\eps R}\bigl(e^R+\Delta_k^{1/2}\#W(R)\bigr).\]

\subsection{Counting the contributing conjugacy classes}\label{counting-classes-overview-subsec}
By Lemmata~\ref{lem:ConjCount}--\ref{twopower2}, the conjugacy classes $[\gamma]\in W(R)$ are ``almost determined'' by their traces,
with no more than $\ll_{\eps}e^{\eps R}$ classes in $W(R)$ of any given trace. On the other hand, $[\gamma]\in W(R)$ implies (cf.\ \eqref{where-trace-is} and the comment under \eqref{tracebound1}) that $\tr(\gamma)\in\mo$, ${|\!\tr(\gamma)|}_1<e^{R+2}$, and ${|\!\tr(\gamma)|}_j<e^2$ for $j\in\{2,3,4\}$. Therefore,
\[\# W(R)\ll_\eps e^{\eps R}\cdot\#\left\{x\in\mo:\text{${|x|}_1<e^{R+2}$ and ${|x|}_j<e^2$ for $j\in\{2,3,4\}$}\right\}.\]
Since the covolume of $\mo$ in $\AA_{\infty}$ is $\Delta_k^{1/2}$, following the volume--covolume principle, one would expect that the count on the right hand side is about $e^R/\Delta_k^{1/2}$. By \cite[Cor.~1]{FHM20}, this is indeed the case as long as $e^R\gg\Delta_k$, reflecting the fact that the lattice $\mo$ is not too skew. Therefore, substituting this upper bound on $\#W(R)$ into the above bound for $\mm(\pi_{\sigma,iT},\Gamma)$, and recalling also \eqref{correctchoice},
\[\mm(\pi_{\sigma,iT},\Gamma)\ll_\eps e^{-2R\sigma+R+5\eps R}\ll_\eps\bigl(\Delta_k^{3/2}(1+|T|)\bigr)^{1-2\sigma+5\eps}.\]
This completes our sketch of the proof of \eqref{needtoprove}.

\section{Archimedean aspects}\label{sec:Arch}

\subsection{Haar measures}\label{sec:HaarArch}
We specify a Haar measure on $G=\SL_2(\RR)$ in terms of the Iwasawa decomposition $G=ANK$. As usual, $A$ is the subgroup of positive diagonal matrices, $N$ is the subgroup of upper triangular unipotent matrices, and $K$ equals $\SO_2(\RR)$. We define Haar measures on $A$ and $N$ by the formulae
\begin{alignat*}{2}
\int_A f(a)\,da&:=\int_\RR f(a(t))\,dt,&\qquad a(t)&:=\begin{pmatrix}e^{t/2} & \\ & e^{-t/2}\end{pmatrix};\\[3pt]
\int_N f(n)\,dn&:=\int_\RR f(n(x))\,dx,& n(x)&:=\begin{pmatrix}1 & x \\ & 1\end{pmatrix}.
\end{alignat*}
We write $dk$ for the Haar probability measure on $K$, and then we put
\begin{equation}\label{haarG}
\int_G f(g)\,dg:=\int_K \int_N \int_A f(ank)\,da\,dn\,dk.
\end{equation}
Then, \cite[Th.~11.2.1]{DE} along with the normalization under \cite[Th.~11.1.3]{DE}
shows that the same measure in terms of the Cartan decomposition reads
\begin{equation}\label{eq:haar_sl2r_cartan}
\int_G f(g)\,dg = 2\pi \int_K \int_0^{\infty} \int_K f\left(k_1a(t)k_2\right) (\sinh t) \,dk_1\,dt\,dk_2.
\end{equation}

On $G=\SL_2(\CC)$ we specify the Haar measure similarly. We start from the Iwasawa decomposition $G=ANK$, where the meaning of
$A$ and $N$ is as before, and $K$ equals $\SU_2(\CC)$. We use the same Haar measure on $A$ as before, on $N$ we take
\[\int_N f(n)\,dn:=\int_{\RR^2} f(n(x+iy))\,dx\,dy,\]
with $n(\cdot)$ as before, and on $K$ we take the Haar probability measure.
We define $dg$ by \eqref{haarG}, and then \cite[Th.~1.7.1]{JL}
shows that
\begin{equation}\label{eq:haar_sl2c_cartan}
\int_G f(g)\,dg = 4\pi\int_K \int_0^{\infty} \int_K f\left(k_1a(t)k_2 \right) (\sinh t)^2 \, dk_1 \,dt \,dk_2.
\end{equation}

\subsection{Orbital integrals}\label{sec:ArchimedeanOI}
Let $G=\SL_2(\KK)$, where $\KK$ is either $\RR$ or $\CC$. We shall write $\rho$ for the degree $[\KK:\RR]$, and ${|\cdot|}_\KK=|\cdot|^\rho$ for the module function of $\KK$. Let $D$ be the subgroup of diagonal matrices, and $T:=D\cap K$ be its maximal compact subgroup. Based on the polar decomposition $D=AT$, we specify a Haar measure on $D$ as the product of the Haar measure $da$ on $A$ and the Haar probability measure on $T$.

Let $\gamma\in\Gr$ be a regular semisimple element with distinct eigenvalues $w,w^{-1}\in\CC^{\times}$. The centralizer of $\gamma$ in $G$, denoted by $G_\gamma$, is conjugate to $D$ or $K$ (but not both). More precisely, if $w\in\KK$, then there are precisely two ways to map $G_\gamma$ to $D$ by an inner automorphism of $G$, and these are connected by the inverse map on $D$ (conjugation by the Weyl element). Otherwise $\KK=\RR$ and $|w|=1$, and there is a unique way to map $G_\gamma$ to $K$ by an inner automorphism of $G$. Depending on the case, we use a conjugation to transport the Haar measure on $D$ or $K$ to $G_\gamma$. As $G_\gamma$ is closed, and both $G$ and $G_{\gamma}$ are unimodular, the coset space $G_\gamma\bs G$ carries a unique right $G$-invariant measure $d\dot{g}$ such that
\[\int_G f(g)\,dg=\int_{G_\gamma\bs G}\left(\int_{G_\gamma} f(hg)\,dh\right)d\dot{g},\qquad f\in C_c(G).\]
Here $\dot g$ abbreviates the coset $G_\gamma g$. We shall think of $g\in G$ as running through a set of representatives, and write $dg$ instead of $d\dot g$ for convenience.

We introduce the \emph{orbital integral} of a compactly supported function $f$ over the conjugacy class of $\gamma$:
\begin{equation}\label{Odef}
\bO(\gamma,f):=\int_{G_\gamma\bs G} f(g^{-1}\gamma g)\,dg,\qquad\gamma\in\Gr,\quad f\in C_c(G).
\end{equation}
The orbital integral has an important conjugation invariance property, which is plausible but which we spell out for clarity.

\begin{lemma}\label{lem1} Let $\gamma,\delta\in\Gr$ be regular semisimple elements. If $\gamma$ and $\delta$ are conjugate in $G$, then
\[\bO(\gamma,f)=\bO(\delta,f),\qquad f\in C_c(G).\]
\end{lemma}

\begin{proof} By assumption, there is an inner automorphism $g\mapsto hgh^{-1}$ of $G$ that maps $\gamma$ to $\delta$, and $G_\gamma$ to $G_\delta$. This automorphism transports the measure on $G_\gamma\bs G$ to the measure on $G_\delta\bs G$, hence
\begin{align*}
\bO(\gamma,f)&=\int_{G_\gamma\bs G} f(g^{-1}\gamma g)\,dg
=\int_{G_\gamma\bs G} f(g^{-1}h^{-1}\delta hg)\,dg\\
&=\int_{G_\gamma\bs G} f(h^{-1}(hg^{-1}h^{-1})\delta(hgh^{-1})h)\,dg\\
&=\int_{G_\delta\bs G} f(h^{-1}g^{-1}\delta gh)\,dg
=\int_{G_\delta\bs G} f(g^{-1}\delta g)\,dg
=\bO(\delta,f).
\end{align*}
In the last step, we used that the measure on $G_\delta\bs G$ is right $G$-invariant. The proof is complete.
\end{proof}

\begin{remark} Conjugation by the Weyl element acts by the inverse map on $D$, hence Lemma~\ref{lem1} implies that
\begin{equation}\label{flipgamma}
\bO(\gamma,f)=\bO(\gamma^{-1},f),\qquad\gamma\in\Dr,\quad f\in C_c(G).
\end{equation}
\end{remark}

We shall need to estimate $\bO(\gamma,f)$ for a given spherical function $f\in C_c(K\bs G/K)$ and varying $\gamma$. As we do this, thanks to Lemma~\ref{lem1}, we will be able to assume that either $\gamma\in D$, or $G=\SL_2(\RR)$ and $\gamma\in\SO_2(\RR)$. In both cases, we shall express $\bO(\gamma,f)$ in terms of the \emph{Harish-Chandra transform}
\begin{equation}\label{Hdef}
\bH(y,f):={|y|}_\KK^{1/2}\int_N f(a(\log y) n)\,dn,\qquad y>0,\quad f\in C_c(K\bs G/K),
\end{equation}
and the \emph{Weyl discriminant}
\[\Delta(\gamma):=(w-w^{-1})^2=\tr(\gamma)^2-4,\qquad\gamma\in G.\]
The usage of the $y$-coordinate in \eqref{Hdef} is justified by the fact that, by our conventions, $d(a(\log y))=d(\log y)=dy/y$, which is also the measure used in the definition of the Mellin transform.

\begin{lemma}\label{lem2} Let $\gamma\in\Dr$ be a regular diagonal element with distinct eigenvalues $w,w^{-1}\in\CC^{\times}$. Then
\[{|\Delta(\gamma)|}_\KK^{1/2}\,\bO(\gamma,f)=\bH(|w|^2,f),\qquad f\in C_c(K\bs G/K).\]
\end{lemma}

\begin{proof}
Using the definition of the measures on $G$, $D$, $D\bs G$, we see that (cf.\ \eqref{haarG} and \eqref{Odef})
\[\bO(\gamma,f)=\int_{D\bs G} f(g^{-1}\gamma g)\,dg=\int_K\int_N f(k^{-1}n^{-1}\gamma n k)\,dn\,dk=\int_N f(n^{-1}\gamma n)\,dn.\]
By \eqref{flipgamma}, we can assume without loss of generality that $\gamma=\diag(w,w^{-1})$. Using also the notation $n=n(z)$, a small calculation gives that the commutator $\gamma^{-1}n^{-1}\gamma n$ equals $n(z-w^{-2}z)$. Therefore, a change of variables yields that
\[\bO(\gamma,f)=\int_N f(\gamma\gamma^{-1}n^{-1}\gamma n)\,dn=\frac{1}{{|1-w^{-2}|}_\KK}\int_N f(\gamma n)\,dn.\]
As $f\in C_c(K\bs G/K)$, we can replace $\gamma$ by $\diag(|w|,|w|^{-1})$ on the right hand side, and hence
\[{|w-w^{-1}|}_\KK\,\bO(\gamma,f)={|w|}_\KK\int_N f(\diag(|w|,|w|^{-1})n)\,dn=\bH(|w|^2,f).\]
As ${|w-w^{-1}|}_\KK$ equals ${|\Delta(\gamma)|}_\KK^{1/2}$, we are done.
\end{proof}

\begin{remark} Incidentally, Lemma~\ref{lem2} implies the following symmetry of the Harish-Chandra transform:
\begin{equation}\label{Hsymmetry}
\bH(y,f)=\bH(y^{-1},f),\qquad y>0,\quad f\in C_c(K\bs G/K).
\end{equation}
\end{remark}

There is a nice analogue of Lemma~\ref{lem2} for elliptic orbital integrals. Before stating this result, we introduce some notation. For every $g\in G$, we define the \emph{height} $H(g)$ as the unique nonnegative number such that
\[g=k_1a(H(g))k_2\]
for some $k_1,k_2\in K$. That is,
\begin{equation}\label{height}\cosh(H(g))=\frac{1}{2}\tr(gg^*),\qquad g\in G.\end{equation}
We shall also write $1+2h(g)$ for the left hand side, that is,
\begin{equation}\label{hdef}h(g):=\frac{1}{4}\tr(gg^*)-\frac{1}{2},\qquad g\in G.\end{equation}
It is clear that $h(g)\geq 0$, since $\cosh(H(g))\geq 1$. The rationale behind these definitions is that a spherical function $f(g)$ can be thought of as a function of $H(g)$ or $h(g)$, whichever is more convenient.

The following identity for elliptic orbital integrals was inspired by a calculation in \cite[\S 10.6]{Iw}, namely by \cite[(10.28)]{Iw} and the surrounding two displays.

\begin{lemma}\label{lem3} Assume that $\KK=\RR$. Let $\gamma\in\SO_2(\RR)$ be a regular element with distinct eigenvalues $e^{\pm i\theta}$ on the unit circle. Then
\begin{equation}\label{OHell}
{|\Delta(\gamma)|}_\KK^{1/2}\,\bO(\gamma,f)=\int_\RR\frac{\bH(e^{r(v)},f)}{1+v^2}\,dv,\qquad f\in C_c(K\bs G/K),
\end{equation}
where $r(v)$ abbreviates $2\sinh^{-1}(v\sin\theta)$.
\end{lemma}

\begin{proof}
Using the definition of the measures on $G$, $K$, $K\bs G$, we see that (cf.\ \eqref{eq:haar_sl2r_cartan} and \eqref{Odef})
\begin{align*}
\bO(\gamma,f)&=\int_{K\bs G} f(g^{-1}\gamma g)\,dg\\
&=2\pi\int_0^\infty\int_K f(k^{-1}a(-t)\gamma a(t)k)\,(\sinh t)\,dk\,dt\\
&=2\pi\int_0^\infty f(a(-t)\gamma a(t))\,(\sinh t)\,dt.
\end{align*}
We also observe that
\[a(-t)\gamma a(t)=
\begin{pmatrix}e^{-t/2} & \\ & e^{t/2}\end{pmatrix}
\begin{pmatrix}\cos\theta & \pm\sin\theta \\ \mp\sin\theta & \cos\theta\end{pmatrix}
\begin{pmatrix}e^{t/2} & \\ & e^{-t/2}\end{pmatrix}=
\begin{pmatrix}\cos\theta & \pm e^{-t}\sin\theta \\ \mp e^t\sin\theta & \cos\theta\end{pmatrix},\]
whence by \eqref{hdef},
\[h(a(-t)\gamma a(t))=\frac{e^{2t}+e^{-2t}}{4}\sin^2\theta+\frac{\cos^2\theta-1}{2}=(\sinh t)^2(\sin\theta)^2.\]
Therefore, if $F\colon[0,\infty)\to\CC$ denotes the unique function satisfying $f(g)=F(h(g))$, then
\begin{equation}\label{Ocalc}
{|\Delta(\gamma)|}_\KK^{1/2}\,\bO(\gamma,f)=4\pi|\sin\theta|\int_0^\infty F((\sinh t)^2(\sin\theta)^2)\,(\sinh t)\,dt.
\end{equation}

We need to show that the right hand sides of \eqref{OHell} and \eqref{Ocalc} are equal. We start by rewriting $\bH(e^{r(v)},f)$ in terms of $F$. In general, for an arbitrary $r\in\RR$, we have by \eqref{hdef}
\[h(a(r)n(x))=h\left(\begin{pmatrix}e^{r/2} & e^{r/2}x\\ & e^{-r/2}\end{pmatrix}\right)=\frac{e^r x^2}{4}+\sinh^2\frac{r}{2},\]
hence by \eqref{Hdef} also
\[\bH(e^r,f)=e^{r/2}\int_\RR F\Bigl(\frac{e^r x^2}{4}+\sinh^2\frac{r}{2}\Bigr)\,dx
=2|\sin\theta|\int_\RR F\Bigl(x^2\sin^2\theta+\sinh^2\frac{r}{2}\Bigr)\,dx.\]
We apply this to $r(v):=2\sinh^{-1}(v\sin\theta)$ in place of $r$, and integrate the resulting expression against the measure $dv/(1+v^2)$:
\begin{align*}
\int_\RR\frac{\bH(e^{r(v)},f)}{1+v^2}\,dv
&=2|\sin\theta|\int_\RR\int_\RR\frac{F(x^2\sin^2\theta+v^2\sin^2\theta)}{1+v^2}\,dx\,dv\\[3pt]
&=8|\sin\theta|\int_0^\infty\int_0^\infty\frac{F(x^2\sin^2\theta+v^2\sin^2\theta)}{1+v^2}\,dx\,dv.
\end{align*}
We make the polar-like change of variables
\[x=(\sinh t)(\cos\phi)\qquad\text{and}\qquad v=(\sinh t)(\sin\phi),\]
where $t>0$ and $0<\phi<\pi/2$. The map $(t,\phi)\mapsto (x,y)$ is a diffeomorphism from $(0,\infty)\times(0,\pi/2)$ to $(0,\infty)^2$ with absolute Jacobian determinant $(\sinh t)(\cosh t)$. Therefore,
\[\int_\RR\frac{\bH(e^{r(v)},f)}{1+v^2}\,dv=8|\sin\theta|\int_0^\infty\int_0^{\pi/2}
\frac{F((\sinh t)^2(\sin\theta)^2)}{1+(\sinh t)^2(\sin\phi)^2}\,(\sinh t)(\cosh t)\,d\phi\,dt.\]
However, for every $t\in\RR$,
\[\int_0^{\pi/2}\frac{\cosh t}{1+(\sinh t)^2(\sin\phi)^2}\,d\phi=\Bigl[\tan^{-1}\left((\cosh t)(\tan\phi)\right)\Bigr]_0^{\pi/2-}=\frac{\pi}{2}.\]
As a result,
\[\int_\RR\frac{\bH(e^{r(v)},f)}{1+v^2}\,dv=4\pi|\sin\theta|\int_0^\infty F((\sinh t)^2(\sin\theta)^2)\,(\sinh t)\,dt.\]
The right hand side is identical to the right hand side of \eqref{Ocalc}, hence we are done.
\end{proof}

\begin{remark} It follows from Lemma~\ref{lem1} combined with \eqref{flipgamma} and \eqref{Ocalc} that
\[\bO(\gamma,f)=\bO(\gamma^{-1},f),\qquad\gamma\in\Gr,\quad f\in C_c(K\bs G/K).\]
\end{remark}

We shall use Lemmata~\ref{lem1}--\ref{lem3} to estimate the relevant archimedean orbital integrals by the corresponding Harish-Chandra transforms.

\begin{lemma}\label{ObyH} Let $\gamma\in\Gr$ be a regular semisimple element. Then
\[{|\Delta(\gamma)|}_\KK^{1/2}\,|\bO(\gamma,f)|\leq\pi\,\sup_{y>0}|\bH(y,f)|,\qquad f\in C_c(K\bs G/K).\]
\end{lemma}

\begin{proof}
By Lemma~\ref{lem1}, we can assume that either $\gamma\in D$, or $G=\SL_2(\RR)$ and $\gamma\in\SO_2(\RR)$. In the first case, the bound follows trivially from Lemma~\ref{lem2}. In the second case, the bound follows from Lemma~\ref{lem3} combined with the triangle inequality for integrals.
\end{proof}

We end this subsection with an elementary estimate to be used in \S\ref{geometricside-sec}.
\begin{lemma}\label{trDeltabounds}
Let $\gamma\in G$. Then $|\!\tr(\gamma)|^2$ and $|\Delta(\gamma)|$ are at most $4\cosh(H(\gamma))$.
\end{lemma}
\begin{proof} Let $\gamma=\left(\begin{smallmatrix}a & b\\ c & d\end{smallmatrix}\right)\in G$, so that $\tr(\gamma)=a+d$ and $\Delta(\gamma)=(a+d)^2-4=(a-d)^2+4bc$. Then $|\!\tr(\gamma)|^2$ and $|\Delta(\gamma)|$ are at most $2(|a|^2+|b|^2+|c|^2+|d|^2)$, which in turn equals $4\cosh(H(\gamma))$ by \eqref{height}.
\end{proof}

\subsection{Spherical functions}\label{sec:Spherical}
For $s\in\CC$, we denote by $\phi_s\in C^\infty(K\bs G/K)$ the elementary spherical function of normalized Casimir eigenvalue $1/4-s^2$. In particular,
\begin{equation}\label{phisym}
\phi_{-s}(g)=\phi_s(g)=\phi_s(g^{-1}).
\end{equation}
The function $\phi_s(g)$ can be realized as the matrix coefficient $\langle\pi_s(g)v,v\rangle$, where $(\pi_s,V_s)$ is the spherical principal series representation of $G$ induced from the character $\diag(y^{1/2},y^{-1/2})\mapsto{|y|}_\KK^s$, and $v\in V_s$ is a $K$-invariant unit vector. In this language, the imaginary axis $s\in i\RR$ corresponds to the tempered spectrum, while the real interval $s\in(-1/2,1/2)$ corresponds to the exceptional spectrum (complementary series). In terms of special functions, we have
\begin{alignat}{2}
\label{phi1}\phi_s(g)&=P_{-1/2+s}(\cosh(H(g))),&\qquad&\KK=\RR;\\[3pt]
\label{phi2}\phi_s(g)&=\frac{\sinh(2s H(g))}{2s\sinh(H(g))},&&\KK=\CC.
\end{alignat}
Here, $P_{-1/2+s}$ denotes the Legendre function, and $H(g)$ is given by \eqref{height}. See \cite[p.~23]{Iw}, \cite[pp.~47, 84, 202]{L}, \cite[(22.6.11.7)]{D}, \cite[p.~433, (30)]{H}, \cite[\S\S 1.5--1.8]{G}, \cite[Prop.~2.4.5]{JL} for more details.

The \emph{spherical transform} of $f\in C_c^\infty (K\bs G/K)$ is given by
\begin{equation}\label{spherical}
\widehat{f}(s):=\int_{G}f(g)\,\phi_{s}(g)\,dg,\qquad s\in\CC.
\end{equation}
Thus $\widehat{f}(s)=\langle\pi_s(f)v,v\rangle$ for a $K$-invariant unit vector $v\in V_s$ as above.
The spherical transform can also be understood as the \emph{normalized Mellin transform} of the Harish-Chandra transform:
\begin{equation}\label{sphericalM}
\widehat{f}(s)=\int_0^\infty{|y|}_\KK^s\,\bH(y,f)\,\frac{dy}{y}=\int_0^\infty y^{\rho s}\,\bH(y,f)\,\frac{dy}{y}.
\end{equation}
Based on this connection, the spherical transform $f\mapsto\widehat{f}$ provides an isomorphism between the convolution algebra $C_c^\infty (K\bs G/K)$ and the usual algebra of even entire functions satisfying a Paley--Wiener type condition. More precisely, let us introduce the notation
\begin{equation}\label{BRdef}
B(R):=\{g\in G:\ H(g)\leq R\},\qquad R>0.
\end{equation}
Using the known inverses of the Mellin and Harish-Chandra transforms (cf.\ \cite[Ch.~5, Th.~3--5]{L}), it is straightforward to verify that the following three conditions are equivalent:
\smallskip
\begin{itemize}
\listsep
\item $f$ is supported in $B(R)$;
\item $\bH(y,f)$ is supported in $[e^{-R},e^R]$;
\item $\widehat{f}(\sigma+i\tau)\ll_N e^{\rho R|\sigma|}(1+|\tau|)^{-N}$ for every $\sigma,\tau\in\RR$ and $N>0$.
\end{itemize}
The inverse spherical transform on the even Paley--Wiener space can be given directly as
\begin{alignat}{2}
\label{inv1}f(g)&=\frac{1}{2\pi}\int_\RR\phi_{i\tau}(g)\,\widehat{f}(i\tau)\, \tau\tanh(\pi\tau)\,d\tau,&\qquad&\KK=\RR;\\[3pt]
\label{inv2}f(g)&=\frac{2}{\pi^2}\int_\RR\phi_{i\tau}(g)\,\widehat{f}(i\tau)\, \tau^2\,d\tau,&&\KK=\CC.
\end{alignat}
The leading constants here are rather ad-hoc as they depend on the normalization of the Haar measure on $G$. In fact \eqref{inv1} is equivalent to the inversion formula for the classical Mehler--Fock transform, as can be seen from \eqref{eq:haar_sl2r_cartan}, \eqref{phi1}, \eqref{spherical}. Similarly, \eqref{inv2} is equivalent to the inversion formula for the classical sine transform, as can be seen from \eqref{eq:haar_sl2c_cartan}, \eqref{phi2}, \eqref{spherical}. See \cite[Ch.~V]{L}, \cite[Ch.~1]{JL}, \cite[\S\S 1.9--1.10]{G} for more details\footnote{We note that \cite[Ch.~V, Th.~6--7]{L} are slightly in error, e.g.\ $\mathbf{S}f(i\tau)$ should be $\mathbf{S}f(2i\tau)$, which is apparently our $\widehat{f}(i\tau)$.}.

In the previous section, we saw how to estimate an orbital integral in terms of the Harish-Chandra transform. Now we can state and prove a similar bound in terms of the spherical transform.

\begin{lemma}\label{ObyS} Let $\gamma\in\Gr$ be a regular semisimple element. Then
\[{|\Delta(\gamma)|}_\KK^{1/2}\,|\bO(\gamma,f)|\leq\int_\RR|\widehat{f}(i\tau)|\,d\tau,\qquad f\in C_c(K\bs G/K).\]
\end{lemma}

\begin{proof} By \eqref{sphericalM} we have, for any $y>0$,
\[\bH(y,f)=\frac{1}{2\pi i}\int_{i\RR}y^{-s}\,\widehat{f}\left(\frac{s}{\rho}\right)ds
=\frac{\rho}{2\pi}\int_\RR y^{-i\tau\rho}\,\widehat{f}(i\tau)\,d\tau.\]
Hence the result follows from Lemma~\ref{ObyH} and the triangle inequality for integrals.
\end{proof}

\subsection{Construction of test functions}
Recall that $G=\SL_2(\KK)$, where $\KK$ is either $\RR$ or $\CC$, and $\rho=[\KK:\RR]$. We shall construct a positive definite test function $F\in C_c^\infty (K\bs G/K)$ with well-controlled support and size, such that the spherical transform $\widehat{F}(s)$ is sufficiently large for certain spectral parameters $s\in\CC$, while the orbital integrals $\bO(\gamma,F)$ are sufficiently small for all regular semisimple elements $\gamma\in\Gr$. These estimates will be crucial in our application of the trace formula. As in the original construction of Sarnak--Xue~\cite[(28)--(29)]{SX}, we shall provide $F$ in the form of $u\ast\widecheck{u}$, where
\[\widecheck{u}(g):=\ov{u(g^{-1})},\qquad g\in G.\]
In terms of the spherical transform, this means that (cf.\ \eqref{phisym} and \eqref{spherical})
\begin{equation}\label{hatF}
\widehat{F}(s)=\widehat{u}(s)\ov{\widehat{u}(\ov{s})},\qquad s\in\CC.
\end{equation}

We fix, once and for all, a smooth function $f\colon K\bs G/K\to[0,\infty)$, which is supported in $B(1/2)$ but is not identically zero. Then, the Harish-Chandra transform $\bH(y,f)$ is nonnegative for every $y>0$, and it is supported in $[e^{-1/2},e^{1/2}]$. Using \eqref{Hsymmetry} and \eqref{sphericalM}, we see that
\[\widehat{f}(s)=\int_0^\infty \cosh(\rho s\log y)\,\bH(y,f)\,\frac{dy}{y},\]
hence $\widehat{f}(s)$ is real for $s\in\RR\cup i\RR$, and
\begin{equation}\label{hatflower}
\widehat{f}(s)\geq\cos(1)\,\widehat{f}(0)>0,\qquad s\in[-1,1]\cup i[-1,1].
\end{equation}
Moreover,
\begin{equation}\label{hatfupper}
\widehat{f}(\sigma+i\tau)\ll_N e^{\rho|\sigma|/2}(1+|\tau|)^{-N},\qquad\sigma,\tau\in\RR,\quad N>0.
\end{equation}
We shall define $F(g)$ in terms of $f(g)$. Our explicit constructions appear in the proofs of the two propositions below, which will be used at non-tempered and tempered archimedean places, respectively.

\begin{proposition}\label{nontempered} Let $R\geq 0$. There exists $u\in C_c^\infty (K\bs G/K)$ such that the positive definite spherical function $F:=u\ast\widecheck{u}$ has the following properties:
\begin{itemize}
\listsep
\item $\widehat{F}(\sigma)\gg e^{2\rho R|\sigma|}$ for every $\sigma\in[-1,1]$;
\item $F$ is supported in $B(2R+2)$, and it satisfies the bound $F\ll 1$;
\item ${|\Delta(\gamma)|}_\KK^{1/2}\,|\bO(\gamma,F)|\ll 1$ for every $\gamma\in\Gr$.
\end{itemize}
The implied constants are absolute.
\end{proposition}

\begin{proof} Let us define $u\in C_c^\infty (K\bs G/K)$ through its spherical transform
\begin{equation}\label{udef}
\widehat{u}(s):=\cosh(\rho Rs)\,\widehat{f}(s)^2,\qquad s\in\CC.
\end{equation}
This function is even, entire, and by \eqref{hatfupper} it satisfies the Paley--Wiener type condition
\[\widehat{u}(\sigma+i\tau)\ll_N e^{\rho(R+1)|\sigma|}(1+|\tau|)^{-N},\qquad\sigma,\tau\in\RR,\quad N>0.\]
Hence indeed $\widehat{u}(s)$ is the spherical transform of a unique $u\in C_c^\infty (K\bs G/K)$, which is supported in $B(R+1)$.
By \eqref{hatF} and a similar reasoning (or by direct calculation), it is clear now that $F:=u\ast\widecheck{u}$ is supported in $B(2R+2)$.
We record the following simple consequence of \eqref{hatF} and \eqref{udef}:
\[\widehat{F}(s)=\cosh(\rho Rs)^2\,\widehat{f}(s)^4,\qquad s\in\CC.\]
Combining this with \eqref{hatflower}, we obtain for every $\sigma\in[-1,1]$ that
\[\widehat{F}(\sigma)\geq\cosh(\rho R\sigma)^2\,\cos(1)^4\,\widehat{f}(0)^4\gg e^{2\rho R|\sigma|}.\]
Similarly, from \eqref{hatfupper} and the spherical inversion formulae \eqref{inv1}--\eqref{inv2}, we obtain for every $g\in G$ that
\[F(g)\leq\int_\RR\widehat{f}(i\tau)^4\,|\tau|^\rho\,d\tau\ll 1.\]
Finally, by \eqref{hatfupper} and Lemma~\ref{ObyS}, we infer for every $\gamma\in\Gr$ that
\[{|\Delta(\gamma)|}_\KK^{1/2}\,|\bO(\gamma,F)|\leq\int_\RR\widehat{f}(i\tau)^4\,d\tau\ll 1.\]
The proof is complete.
\end{proof}

\begin{proposition}\label{tempered}Let $t\in\RR$. There exists $u\in C_c^\infty (K\bs G/K)$ such that the positive definite spherical function $F:=u\ast\widecheck{u}$ has the following properties:
\begin{itemize}
\listsep
\item $\widehat{F}(i\tau)\gg 1$ for every $\tau\in[t-1,t+1]$;
\item $F$ is supported in $B(2)$, and it satisfies the bound $F\ll(1+|t|)^\rho$;
\item ${|\Delta(\gamma)|}_\KK^{1/2}\,|\bO(\gamma,F)|\ll 1$ for every $\gamma\in\Gr$.
\end{itemize}
The implied constants are absolute.
\end{proposition}

\begin{proof} Let us define $u\in C_c^\infty (K\bs G/K)$ through its spherical transform
\begin{equation}\label{udef2}
\widehat{u}(s):=\widehat{f}(s-it)^2+\widehat{f}(s+it)^2,\qquad s\in\CC.
\end{equation}
This function is even, entire, and by \eqref{hatfupper} it satisfies the Paley--Wiener type condition
\[\widehat{u}(\sigma+i\tau)\ll_{t,N} e^{\rho|\sigma|}(1+|\tau|)^{-N},\qquad\sigma,\tau\in\RR,\quad N>0.\]
Hence indeed $\widehat{u}(s)$ is the spherical transform of a unique $u\in C_c^\infty (K\bs G/K)$, which is supported in $B(1)$.
By \eqref{hatF} and a similar reasoning (or by direct calculation), it is clear now that $F:=u\ast\widecheck{u}$ is supported in $B(2)$.
We record the following simple consequence of \eqref{hatF} and \eqref{udef2}:
\[\widehat{F}(s)=\left(\widehat{f}(s-it)^2+\widehat{f}(s+it)^2\right)^2,\qquad s\in\CC.\]
Combining this with \eqref{hatflower}, we obtain for every $\tau\in[t-1,t+1]$ that
\[\widehat{F}(i\tau)\geq\cos(1)^4\,\widehat{f}(0)^4\gg 1.\]
Similarly, from \eqref{hatfupper} and the spherical inversion formulae \eqref{inv1}--\eqref{inv2}, we obtain for every $g\in G$ that
\[F(g)\leq\int_\RR\widehat{f}(i\tau-it)^4\,|\tau|^\rho\,d\tau\ll (1+|t|)^\rho.\]
In the last step, we used that $|\tau|^\rho$ is less than the product of $(1+|\tau-t|)^\rho$ and $(1+|t|)^\rho$.
Finally, by \eqref{hatfupper} and Lemma~\ref{ObyS}, we infer for every $\gamma\in\Gr$ that
\[{|\Delta(\gamma)|}_\KK^{1/2}\,|\bO(\gamma,F)|\leq 4\int_\RR\widehat{f}(i\tau)^4\,d\tau\ll 1.\]
The proof is complete.
\end{proof}

\section{Non-archimedean aspects}\label{sec:NonArch}

\subsection{Haar measures and orbital integrals}\label{sec:HaarNarch}
Let $F$ be a non-archimedean local field with ring of integers $\mo$, maximal ideal $\mpr\subset\mo$, and uniformizer $\varpi\in\mpr$ (which we fix throughout). As usual, we define $v_{\mpr}\colon\mo\to\NN\cup\{\infty\}$ by setting $v_{\mpr}(x)=r$ for $x\in\varpi^r\mo^{\times}$ and $v_{\mpr}(0)=\infty$. We write $q:=\#(\mo / \mpr)$ for the size of the residue field, and we normalize the multiplicative valuation ${|\cdot|}_F$ on $F$ so that ${|\varpi|}_F=q^{-1}$. We set \[G:=\SL_2(F)\qquad\text{and}\qquad K:=\SL_2(\mo),\]
and we normalize the Haar measure on $G$ so that $K$ has measure $1$.

A regular semisimple element $\gamma\in\Gr$ has two distinct eigenvalues $w,w^{-1}\in\ov{F}^\times$, and we shall distinguish between the cases $w\in F$ ($\gamma$ is split) and $w\notin F$ ($\gamma$ is elliptic). In both cases, we write $E:=F(w)$ for the splitting field of $\gamma$, and we fix some $\tau\in\GL_2(E)$ such that $\tau^{-1}\gamma\tau=\diag(w,w^{-1})$. In the elliptic case, we can and we shall assume that $\tau$ is of the form $\left(\begin{smallmatrix}x & \ov{x}\\ y & \ov{y}\end{smallmatrix}\right)$, where the upper bar indicates the action of the nontrivial element of $\Gal(E/F)$.

We specify a Haar measure on the centralizer $G_{\gamma}\leq G$ of $\gamma$ as follows. If $\gamma$ is split, then
\[G_{\gamma}=\tau\left\{\diag(a,a^{-1}):a\in F^{\times}\right\}\tau^{-1},\]
which has an inner direct product decomposition $G_\gamma=\Lambda_\gamma T_\gamma$ with
\begin{align*}
\Lambda_\gamma&:=\tau \left\{\diag(\varpi^n,\varpi^{-n}):n\in\ZZ\right\} \tau^{-1},\\[3pt]
T_{\gamma}&:=\tau\left\{\diag(a,a^{-1}):a\in\mo^{\times}\right\}\tau^{-1}.
\end{align*}
We observe that the lattice $\Lambda_\gamma\leq G_\gamma$ and the compact open subgroup $T_\gamma\leq G_\gamma$ are uniquely determined by $\gamma$.
On $\Lambda_\gamma$ we take the counting measure, on $T_\gamma$ we take the Haar probability measure, and on $G_\gamma$ we take the product of these two measures.
If $\gamma$ is elliptic, then
\[G_{\gamma}=G\cap\tau\left\{\diag(a,a^{-1}):a\in E^\times\right\}\tau^{-1}.\]
That is, $G_\gamma$ consists of the matrices $g=\tau\cdot\diag(a,a^{-1})\cdot\tau^{-1}$ such that $a\in E^\times$ and $\ov{g}=g$.
Using the relation $\ov{\tau}=\tau\left(\begin{smallmatrix} & 1\\ 1 & \end{smallmatrix}\right)$, a quick calculation reveals that
$\ov{g}=g$ is equivalent to $a\ov{a}=1$. On the other hand, a multiplicative valuation of $E$ is invariant under $\Gal(E/F)$, hence $a\ov{a}=1$ can only hold when $a$ is a unit in $E$. We conclude that $G_\gamma$ has the transparent description
\[G_{\gamma}=\tau\left\{\diag(a,\ov{a}):\text{$a\in\mo_E^\times$ and $a\ov{a}=1$}\right\}\tau^{-1}.\]
In particular, $G_\gamma$ is a compact group, and we take the Haar probability measure on it.

Now we can repeat what we did in the archimedean situation. As $G_\gamma$ is closed, and both $G$ and $G_{\gamma}$ are unimodular, the coset space $G_\gamma\bs G$ carries a unique right $G$-invariant measure $d\dot{g}$ such that
\[\int_G f(g)\,dg=\int_{G_\gamma\bs G}\left(\int_{G_\gamma} f(hg)\,dh\right)d\dot{g},\qquad f\in C_c(G).\]
Here $\dot g$ abbreviates the coset $G_\gamma g$. We shall think of $g\in G$ as running through a set of representatives, and write $dg$ instead of $d\dot g$ for convenience. When $\gamma$ is split, we specify a right $G$-invariant measure on $\Lambda_\gamma\bs G$ in the same way, using the already chosen measures on $\Lambda_\gamma$ and $G$.

We introduce the \emph{orbital integral} of a compactly supported function $f$ over the conjugacy class of $\gamma$:
\[\bO(\gamma,f):=\int_{G_\gamma\bs G} f(g^{-1}\gamma g)\,dg,\qquad\gamma\in\Gr,\quad f\in C_c(G).\]

Just as in the case of archimedean local fields, we have the following.

\begin{lemma}\label{lem1_nona} Let $\gamma,\delta\in\Gr$ be regular semisimple elements. If $\gamma$ and $\delta$ are conjugate in $G$, then
\[\bO(\gamma,f)=\bO(\delta,f),\qquad f\in C_c(G).\]
\end{lemma}
\begin{proof}
Same as that of Lemma~\ref{lem1}.
\end{proof}

The next lemma computes the orbital integral of the characteristic function of a compact open subgroup $U\leq G$. With later applications in mind, it is convenient to formulate it in a more general form which also includes a normalizer of $U$.

\begin{lemma}\label{lemma:orbital_integral_to_counting} Let $U\leq V\leq G$ be two compact open subgroups such that $V$ normalizes $U$. Let $\gamma\in\Gr$ be a regular semisimple element. If $\gamma$ is split, then
\[\bO(\gamma,\bchar_U)=\vol(V)\cdot\#\left\{g\in \Lambda_\gamma \bs G / V : \gamma gU= gU\right\}.\]
If $\gamma$ is elliptic, then
\[\bO(\gamma,\bchar_U)=\vol(V)\cdot\#\left\{g\in G/V:\gamma g U= gU\right\}.\]
\end{lemma}

\begin{proof}
In both cases, the result follows by a simple calculation using the left $G_\gamma$-invariance and the right $V$-invariance of the function
\[f(g):=\bchar_{U}(g^{-1}\gamma g),\qquad g\in G,\]
along with the definition of the measures on $G_\gamma$ and $G_\gamma\bs G$ (as well as $\Lambda_\gamma$ and $\Lambda_\gamma\bs G$, if $\gamma$ is split).

If $\gamma\in\Gr$ is split, then
\[\bO(\gamma,\bchar_U) = \int_{G_{\gamma} \bs G} f(g)\,dg = \int_{\Lambda_\gamma \bs G} f(g)\,dg =
\vol(V) \sum_{g\in\Lambda_\gamma \bs G / V} \bchar_{\gamma g U = g U}.\]
The last equation can be checked by noting that if $g$ runs through a set of representatives for $\Lambda_\gamma \bs G / V$,
the disjoint union of the corresponding left cosets $g V$ form a set of representatives for $\Lambda_\gamma \bs G$.

If $\gamma\in\Gr$ is elliptic, then
\[\bO(\gamma,\bchar_U) = \int_{G_{\gamma} \bs G} f(g)\,dg = \int_G f(g)\,dg =
\vol(V) \sum_{g\in G/V} \bchar_{\gamma g U = g U}.\]
The proof is complete.
\end{proof}

\subsection{Congruence subgroups and the main estimate}
For a nonzero ideal $\mn\subset\mo$, we introduce the congruence subgroups
\begin{align}
\label{K0} K_0(\mn)&:=\left\{\begin{pmatrix}a & b \\ c & d\end{pmatrix}\in K: c\in\mn \right\},\\[3pt]
\label{K1} K_1(\mn)&:=\left\{\begin{pmatrix}a & b \\ c & d\end{pmatrix}\in K: a-d,b,c\in\mn\right\}.
\end{align}
Note that $K_0(\mn)$ and $K_1(\mn)$ are (closely related to) the local counterparts of the classical congruence subgroups $\Gamma_0(n)$ and $\Gamma(n)$ (not $\Gamma_1(n)$).

The next two results are well-known to experts, but for convenience we provide short proofs.

\begin{lemma}\label{grouplemma1} For a nonzero ideal $\mn\subset\mo$, we have
\[K_1(\mn)=\bigcap_{k\in K}k^{-1}K_0(\mn)k.\]
\end{lemma}

\begin{proof} $K_1(\mn)$ is a normal subgroup of $K$, because it is the kernel of the reduction map $K\to\PSL_2(\mo/\mn)$. Therefore, the left hand side is contained in the right hand side. Now let $g=\left(\begin{smallmatrix}a & b \\ c & d\end{smallmatrix}\right)$ be an element of the right hand side. Then $kgk^{-1}$ lies in $K_0(\mn)$ for every $k\in K$. Applying this for the matrices $\left(\begin{smallmatrix}1 & \\ & 1\end{smallmatrix}\right)$, $\left(\begin{smallmatrix} & -1 \\ 1 & \end{smallmatrix}\right)$, $\left(\begin{smallmatrix}1 & \\ 1 & 1\end{smallmatrix}\right)$ in the role of $k$, we obtain that the lower left entries $c$, $-b$, $a-b+c-d$ lie in $\mn$. Hence $g\in K_1(\mn)$, and we are done.
\end{proof}

\begin{lemma}\label{grouplemma2} Let $r$ be a positive integer, and let $t$ be the size of the $2$-torsion of $(\mo/\mpr^r)^\times$. Then
\[[K:K_0(\mpr^r)]=q^r(1+q^{-1})\qquad\text{and}\qquad[K:K_1(\mpr^r)]=q^{3r}(1-q^{-2})/t.\]
\end{lemma}

\begin{proof} First, we verify the identity for $[K:K_0(\mpr^r)]$. By an induction argument, it suffices to show that
\begin{equation}\label{congruenceindex}
[K_0(\mpr^{r-1}):K_0(\mpr^r)]=\begin{cases}q+1,&\text{if $r=1$,}\\
q,&\text{if $r\geq 2$.}\end{cases}
\end{equation}
Two matrices $\left(\begin{smallmatrix}a & b \\ c & d\end{smallmatrix}\right),\left(\begin{smallmatrix}a' & b' \\ c' & d'\end{smallmatrix}\right)\in K$ lie in the same left coset of $K_0(\mpr^r)$ if and only if $ac'-a'c\in\mpr^r$. Using this relation, one obtains a set of representatives for $K_0(\mpr^{r-1})/K_0(\mpr^r)$ by picking the matrices $\left(\begin{smallmatrix}1 & \\ c & 1\end{smallmatrix}\right)$, where $c$ runs through a set of representatives for $\mpr^{r-1}/\mpr^r$, and also picking the additional matrix $\left(\begin{smallmatrix} & -1 \\ 1 & \end{smallmatrix}\right)$ when $r=1$. This proves \eqref{congruenceindex}.

Now, we verify the identity for $[K:K_1(\mpr^r)]$. The quotient group $K_0(\mpr^r)/K_1(\mpr^r)$ is isomorphic to the group of upper triangular matrices in $\PSL_2(\mo/\mpr^r)$. The latter group has cardinality $\#(\mo/\mpr^r)\cdot\#(\mo/\mpr^r)^\times/t$, hence by the already established identity for $[K:K_0(\mpr^r)]$,
\[[K:K_1(\mpr^r)]=[K:K_0(\mpr^r)]\cdot[K_0(\mpr^r):K_1(\mpr^r)]=q^r(1+q^{-1})\cdot q^{2r}(1-q^{-1})/t.\]
The proof is complete.
\end{proof}

As in the archimedean case, we introduce the \emph{Weyl discriminant}
\[\Delta(\gamma):=(w-w^{-1})^2=\tr(\gamma)^2-4,\qquad\gamma\in G.\]
It will be convenient to use the notation
\[\nu(\gamma):=v_\mpr(\Delta(\gamma))/2,\qquad\gamma\in\Gr.\]
Let $\Kr:=K\cap\Gr$, and define, for $r\in\NN$ and $j\in\{0,1\}$, the functions $w_j^r\colon\Kr\to\RR_{\geq 0}$ as
\begin{align}
\label{w0-def}w_0^r(\gamma)&:=\begin{cases}
q^{\min(\nu(\gamma),\lfloor r/2\rfloor)},& \text{if $\gamma$ is split,}\\[2pt]
q^{\lfloor r/2\rfloor} \bchar_{\mpr^r\mid\Delta(\gamma)},& \text{if $\gamma$ is elliptic;}
\end{cases}\\[6pt]
\label{w1-def}w_1^r(\gamma)&:=q^{2r}\bchar_{\mpr^{2r}\mid\Delta(\gamma)}.
\end{align}
The rest of this section is devoted to the proof of the following estimate on non-archimedean orbital integrals.

\begin{proposition}\label{proposition:orbital_integral_in_nonarch} For $\gamma\in\Kr$, $r\in\NN$, $j\in\{0,1\}$, we have
\begin{equation}\label{nonarchbound}
{|\Delta(\gamma)|}_F^{1/2}\,\bO(\gamma,\bchar_{K_j(\mpr^r)})\ll q^{-\lambda/2} \vol(K_j(\mpr^r)) w_j^r(\gamma),
\end{equation}
where $\lambda=0$ when $E/F$ is unramified, and $\lambda=1$ when $E/F$ is ramified. The implied constant is absolute. Moreover, if ${|\Delta(\gamma)|}_F=1$, then
\begin{equation}\label{nonarchbound2}
\bO(\gamma,\bchar_{K_j(\mpr^r)}) \leq \vol(K_j(\mpr^r)).
\end{equation}
\end{proposition}

\subsection{The Bruhat--Tits tree}
An important tool to reduce the proof of Proposition~\ref{proposition:orbital_integral_in_nonarch} to a graph theoretic enumeration is the Bruhat--Tits tree, of which we recall some important properties from \cite[\S 5]{K}. The \emph{Bruhat--Tits tree} is a graph $\cX_F$ with vertex set $\PGL_2(F)/\PGL_2(\mo)$. Two vertices $u,v\in\cX_F$ are connected by an edge if and only if they can be represented by $g,h\in\GL_2(F)$ such that $g^{-1}h=\diag(\varpi,1)$. Then $\cX_F$ is a $(q+1)$-regular tree, which we shall also regard as a simplicial complex with vertices and edges being the zero- and one-dimensional simplices, respectively. We shall write $[u,v]$ for the unique path joining two vertices $u,v\in\cX_F$.

We equip $\cX_F$ with the geodesic metric by assigning length $1$ to each edge, and we denote by $\dist_F(U,V)$ the distance of two non-empty subsets $U,V\subset\cX_F$. The group $\GL_2(F)$ acts on the left by isometries; the action is transitive on the vertex set. The action of $G=\SL_2(F)$ decomposes the vertex set into two orbits, with no edges between vertices from the same orbit. Infinite paths are called \emph{apartments}, and since $\GL_2(F)$ also acts transitively on the set of apartments, we now select a particular one as follows. The \emph{standard apartment} $\cA_0$ is the infinite path with vertices $v_n:=\diag(\varpi^n,1)\PGL_2(\mo)$ for $n\in\ZZ$. We record that $\{v_n:\text{$n\in\ZZ$ even}\}$ lies in one $G$-orbit, and $\{v_n:\text{$n\in\ZZ$ odd}\}$ lies in the other $G$-orbit.

A finite Galois extension $E/F$ admits a similar Bruhat--Tits tree $\cX_E$ with vertex set $\PGL_2(E)/\PGL_2(\mo_E)$. The Galois group $\Gal(E/F)$ acts on $\cX_E$ by isometries, and if we multiply the distances in $\cX_F$ by the ramification degree of $E/F$, then $\cX_F$ gets contained in the set of fixed points as a metric subspace: $\cX_F\subset\cX_E^{\Gal(E/F)}$. By a rather deep result of Rousseau~\cite[Ch.~V]{R}, we have equality here unless $E/F$ is wildly ramified. A shorter proof was given by Prasad~\cite{P}, but it still relies on a lot of background theory, including earlier chapters of Rousseau's unpublished thesis. These references deal with general Bruhat--Tits buildings and valued fields. For our special situation we can give a quick, direct proof.

\begin{lemma}\label{lemma:BTRP} Let $E/F$ be a finite Galois extension. If $E/F$ is unramified or tamely ramified, then
\[\cX_F=\cX_E^{\Gal(E/F)}.\]
\end{lemma}

\begin{proof} We need to show that every point $v\in\cX_E^{\Gal(E/F)}$ lies in $\cX_F$. We claim that we can restrict to the vertices of $\cX_E$. Indeed, if $v$ is not a vertex of $\cX_E$, then it lies on a unique edge $e$ of $\cX_E$. Either $e$ lies in $\cX_F$, or the unique path from $v$ to $\cX_F$ starts with a segment of $e$. In both cases, we see that $e\subset\cX_E^{\Gal(E/F)}$, justifying our claim.

Let $\mPr$ denote the maximal ideal of $\mo_E$, and let $\varpi_E\in\mPr$ be a uniformizer. The vertices of $\cX_E$ can be written in the Iwasawa form
\begin{equation}\label{eq:vertexIwasawa}
v=\begin{pmatrix} \varpi_E^n & x \\ & 1 \end{pmatrix} \PGL_2(\mo_E),\qquad n\in\ZZ,\quad x\in E.
\end{equation}
The integer $n\in\ZZ$ and the image of $x$ in $E/\mPr^n$ are uniquely determined by $v$, independently of the choice of $\varpi_E$. Indeed, a simple calculation shows that for $x,x'\in E$ we have
\begin{equation}\label{eq:equiv_BTRP}
\begin{pmatrix} \varpi_E^n & x \\ & 1 \end{pmatrix} \PGL_2(\mo_E)=\begin{pmatrix} \varpi_E^n & x' \\ & 1 \end{pmatrix} \PGL_2(\mo_E)
\qquad \Longleftrightarrow \qquad x-x'\in\mPr^n.
\end{equation}

First we consider the case when $E/F$ is unramified; we choose $\varpi_E$ to be the uniformizer $\varpi$ of $F$. We claim that every vertex $v\in\cX_E^{\Gal(E/F)}$ is a vertex of $\cX_F$. Writing $v$ in the form \eqref{eq:vertexIwasawa}, and using \eqref{eq:equiv_BTRP},
we see the following: $x-x^\sigma\in\mPr^n$ holds for all $\sigma\in\Gal(E/F)$, and we need to show that $x-\xi\in\mPr^n$ holds for a suitable $\xi\in F$. The residue field $\mo_E/\mPr$ is represented by the set $A$ consisting of zero and the roots of unity in $E$ of order coprime to $q$. We have a unique Teichm\"uller representation
\[x=\sum_{m\in\ZZ}a_m\varpi^m,\]
where $a_m\in A$, and $a_m=0$ for $m$ sufficiently small. The Galois group acts on $x$ by acting on the digits $a_m$, hence the condition
$x-x^\sigma\in\mPr^n$ translates to $a_m=a_m^\sigma$ for all $m<n$. We conclude that $a_m\in F$ for all $m<n$, and hence
\[\xi:=\sum_{m<n}a_m\varpi^m\]
is an element of $F$ with the required property $x-\xi\in\mPr^n$.

Now we consider the case when $E/F$ is tamely and fully ramified; the degree $e:=[E:F]$ is coprime to $q$. We claim that every vertex $v\in\cX_E^{\Gal(E/F)}$ lies on an edge of $\cX_F$. Writing $v$ in the form \eqref{eq:vertexIwasawa}, and combining \eqref{eq:equiv_BTRP}
with the fact that the Galois translates of $\varpi_E$ are uniformizers, we see (as in the unramified case) that $x-x^\sigma\in\mPr^n$ holds for all $\sigma\in\Gal(E/F)$. The Galois average
\[\xi:=\frac{1}{e}\cdot\sum_{\sigma\in\Gal(E/F)}x^\sigma\]
clearly lies in $F$, and it satisfies
\[x-\xi=\frac{1}{e}\cdot\sum_{\sigma\in\Gal(E/F)}(x-x^\sigma)\in\mPr^n.\]
Therefore, using \eqref{eq:equiv_BTRP} again,
\[v=\begin{pmatrix} \varpi_E^n & \xi \\ & 1 \end{pmatrix} \PGL_2(\mo_E).\]
If $n$ is divisible by $e$, then $\varpi_E^n$ can be replaced by $\varpi^{n/e}$, and we conclude that $v$ is a vertex of $\cX_F$. Otherwise, $v$ is not a vertex of $\cX_F$, but it lies on a path of $\cX_E$ that is identified with an edge of $\cX_F$:
\[v\in\left[\begin{pmatrix} \varpi^{\lfloor n/e \rfloor} & \xi \\ & 1 \end{pmatrix} \PGL_2(\mo_E), \begin{pmatrix} \varpi^{\lceil n/e \rceil} & \xi \\ & 1 \end{pmatrix} \PGL_2(\mo_E)\right].\]

Finally, in the general tamely ramified case, let $F_1$ be the maximal unramified subextension of $E/F$. Then noting that $F_1/F$ is Galois, and combining the two cases above,
\[\cX_F = \cX_{F_1}^{\Gal(F_1/F)} = \left(\cX_E^{\Gal(E/F_1)}\right)^{\Gal(F_1/F)} = \cX_E^{\Gal(E/F)}.\]
The proof is complete.
\end{proof}

For the purpose of our counting arguments, we introduce some technical notations:
\begin{alignat*}{2}
G_U&:=\{g\in G: \text{$gu=u$ for all $u\in U$}\},&\qquad& U\subset\cX_F,\\[3pt]
\cB_F(U,r)&:=\{x\in\cX_F:\dist_F(x,U)\leq r\},&&\emptyset\neq U\subset\cX_F,\quad r\geq 0.
\end{alignat*}

\begin{lemma}\label{lemma:subgroups_characterized_via_fixed_points}
For $r\in\NN$ we have\
\[K_0(\mpr^r)=G_{[v_{-r},v_0]}\qquad\text{and}\qquad K_1(\mpr^r)=G_{\cB_F(v_0,r)}.\]
\end{lemma}

\begin{proof} It is straightforward to check that $K=G_{\{v_0\}}$, and hence
\begin{equation}\label{K0result}
K_0(\mpr^r)=\diag(\varpi^{-r},1)K\diag(\varpi^r,1)\cap K=G_{\{v_{-r},v_0\}}=G_{[v_{-r},v_0]}.
\end{equation}
Combining this result with Lemma~\ref{grouplemma1}, we get that
\[K_1(\mpr^r)=\bigcap_{k\in K}G_{[kv_{-r},v_0]}=G_U,\]
where $U$ is the union of the paths $[kv_{-r},v_0]$ for $k\in K$. It remains to prove that $U=\cB_F(v_0,r)$. Equivalently, every path of length $r$ starting at $v_0$ is of the form $[kv_{-r},v_0]$ with $k\in K$. We could deduce this statement from the transitivity of $\GL_2(F)$ on the set of apartments, but we prefer an argument based on counting. We can assume that $r\geq 1$, in which case the number of paths of length $r$ starting at $v_0$ equals $(q+1)q^{r-1}$. By \eqref{K0result} and Lemma~\ref{grouplemma2}, the number of paths $[kv_{-r},v_0]$ with $k\in K$ also equals $(q+1)q^{r-1}$, so we are done.
\end{proof}

The next lemma is the heart of our counting arguments. We state it in a stronger form than needed, partly for completeness, partly to underline that our counting is rather precise.

\begin{lemma}\label{lemma:fixed_points} Consider the set of fixed points $\cX_F^\gamma$ of some $\gamma\in\Kr$ acting on $\cX_F$.
\begin{itemize}
\listsep
\item[(a)] If $\gamma$ is split, then $\nu(\gamma)\in\NN$ and $\cX_F^\gamma=\cB_F(\cA,\nu(\gamma))$ for a suitable apartment $\cA\subset\cX_F$.
\item[(b)] If $\gamma$ is elliptic and the splitting extension $E/F$ is unramified, then $\nu(\gamma)\in\NN$ and $\cX_F^\gamma=\cB_F(v,\nu(\gamma))$ for a suitable vertex $v\in\cX_F$.
\item[(c)] If $\gamma$ is elliptic and the splitting extension $E/F$ is ramified, then $\cX_F^\gamma\subset\cB_F(e,\nu(\gamma)-\frac12)$ for a suitable edge $e\subset\cX_F$, with equality in the tamely ramified case. In addition, $\nu(\gamma)\in\frac12+\frac12\NN$, and $\nu(\gamma)\in\frac12+\NN$ in the tamely ramified case.
\end{itemize}
\end{lemma}

\begin{proof} (a) The eigenvalues $w,w^{-1}\in F^\times$ are units, because their sum $\tr(\gamma)$ is an integer. So we can assume, without loss of generality, that $\gamma=\diag(w,w^{-1})$ with $w\in\mo^\times$. Then we claim that $\cX_F^\gamma=\cB_F(\cA_0,\nu(\gamma))$, i.e.,
\[\gamma v=v\qquad\Longleftrightarrow\qquad\dist(v,\cA_0)\leq v_{\mpr}(w-w^{-1}),\qquad v\in\cX_F.\]
It suffices to show this equivalence for the vertices of $\cX_F$. Indeed, if $v$ is not a vertex, then it lies on a unique edge $e$. Either $e$ lies in $\cA_0$, or the unique path from $v$ to $\cA_0$ starts with a segment of $e$. In both cases, we see that $v\in\cX_F^\gamma$ is equivalent to $e\subset\cX_F^\gamma$, while $v\in\cB_F(\cA_0,\nu(\gamma))$ is equivalent to $e\subset\cB_F(\cA_0,\nu(\gamma))$.

The vertices of $\cX_F$ can be written in the Iwasawa form
\[v=\begin{pmatrix}y&\\&1\end{pmatrix}\begin{pmatrix}1&x\\&1\end{pmatrix}v_0,\qquad y\in F^\times,\quad x\in F.\]
Keeping in mind that $\gamma=\diag(w,w^{-1})\in\GL_2(\mo)$, we infer
\begin{align*}\gamma v=v
&\qquad\Longleftrightarrow\qquad
\begin{pmatrix}1&-x\\& 1\end{pmatrix}\begin{pmatrix}w&\\&w^{-1}\end{pmatrix}\begin{pmatrix}1&x\\& 1\end{pmatrix}\in\GL_2(\mo)\\[6pt]
&\qquad\Longleftrightarrow\qquad -v_{\mpr}(x)\leq v_{\mpr}(w-w^{-1}).
\end{align*}
Hence it suffices to show that
\begin{equation}\label{disteq}
\dist(v,\cA_0)=\max(0,-v_{\mpr}(x)).
\end{equation}
For this purpose, we introduce the notation
\[\tilde v_n:=\begin{pmatrix}y&\\&1\end{pmatrix}\begin{pmatrix}1&x\\&1\end{pmatrix}v_n
=\begin{pmatrix}\varpi^n y&\\&1\end{pmatrix}\begin{pmatrix}1&\varpi^{-n}x\\&1\end{pmatrix}v_0,\qquad n\in\ZZ.\]
Then $v=\tilde v_0$, and we see that $\tilde v_n\in\cA_0$ is equivalent to $n\leq v_{\mpr}(x)$. Hence the unique path from $v$ to $\cA_0$ has vertex set
\[\{\tilde v_n:n\in\ZZ\text{ and }\min(0,v_{\mpr}(x))\leq n\leq 0\},\]
and \eqref{disteq} follows. The proof of part (a) is complete.

Before turning to the proof of parts (b) and (c), let us make some initial observations. The ratio of $v_\mPr(\Delta(\gamma))$ and $v_{\mpr}(\Delta(\gamma))$ equals the ramification degree of $E/F$. Hence from part (a), or rather its proof, it follows that
\begin{equation}\label{eq:part_a_applied}
\cX_E^{\gamma} = \begin{cases}
\cB_E(\tau \cA_0,\nu(\gamma)),&\text{if $E/F$ is unramified,}\\
\cB_E(\tau \cA_0,2\nu(\gamma)),& \text{if $E/F$ is ramified.}
\end{cases}
\end{equation}
We also see for any $n\in\ZZ$ that
\[
\ov{\tau v_n} = \ov{\tau} v_n = \tau \begin{pmatrix}
& 1 \\ 1 &
\end{pmatrix} \begin{pmatrix}
\varpi^n \\ & 1
\end{pmatrix} \PGL_2(\mo_E) = \tau \begin{pmatrix}
1 \\ & \varpi^n
\end{pmatrix} \begin{pmatrix}
& 1 \\ 1 &
\end{pmatrix} \PGL_2(\mo_E) = \tau v_{-n}.
\]
Therefore, the generator of $\Gal(E/F)$ flips $\tau\cA_0$ from end to end with a unique fixed point $\tau v_0$.

(b) Since $E/F$ is an unramified quadratic extension,
\[\nu(\gamma)=\frac{1}{2}v_\mpr(\Delta(\gamma))=\frac{1}{2}v_\mPr(\Delta(\gamma))=v_\mPr(w-\ov w)\in\NN.\]
Moreover, $\cX_E$ is a $(q^2+1)$-regular tree, which contains $\cX_F$ as a spanned subtree. By Lemma~\ref{lemma:BTRP}, we see that
$\tau\cA_0\cap\cX_F$ consists of a single vertex $v$ of $\cX_F$, and every point of $\cX_F$ is connected to $\tau\cA_0$ through $v$. Using also \eqref{eq:part_a_applied}, we obtain
\[\cX_F^{\gamma}=\cX_E^{\gamma}\cap\cX_F=\cB_E(\tau\cA_0,\nu(\gamma))\cap \cX_F = \cB_F(v,\nu(\gamma)).\]
(c) Since $E/F$ is a ramified quadratic extension, $\Gal(E/F)$ acts trivially on the residue field $\mo_E/\mPr$, hence
\[2\nu(\gamma)=v_\mpr(\Delta(\gamma))=\frac{1}{2}v_\mPr(\Delta(\gamma))=v_\mPr(w-\ov w)\in 1+\NN.\]
Writing $w$ as $a+b\varpi_E$ with $a,b\in\mo$, we see that $v_\mPr(w-\ov w)$ has the same parity as $v_\mPr(\varpi_E-\ov{\varpi_E})$,
therefore $2\nu(\gamma)$ is odd in the tamely ramified case. Moreover, $\cX_E$ is a $(q+1)$-regular tree, which contains $\cX_F$ such that every edge of $\cX_F$ is the union of two edges of $\cX_E$. Let $u\in\tau\cA_0$ and $v\in\cX_F$ be such that $\dist_E(u,v)$ is minimal. It is clear that $u$ and $v$ are vertices of $\cX_E$ fixed by $\Gal(E/F)$. In particular, $u=\tau v_0$. By Lemma~\ref{lemma:BTRP}, we also have $u=v$ in the tamely ramified case. Clearly, $u$ is not a vertex of $\cX_F$, because its neighbors in $\tau\cA_0$ are not fixed by $\Gal(E/F)$. Similarly, $v$ is not a vertex of $\cX_F$, either because $u=v$, or because the neighbor of $v$ in the path $[u,v]$ lies outside $\cX_F$. So $v$ is the midpoint of an edge $e$ of $\cX_F$. Now every point $x\in\cX_F$ is connected to $\tau\cA_0$ through the path $[u,v]$, hence
\[
\dist_E(x,\tau\cA_0)=\dist_E(x,v)+\dist_E(v,u)\geq \dist_E(x,v)=2\dist_F(x,v),
\]
with equality in the tamely ramified case. Using also \eqref{eq:part_a_applied}, we obtain
\[
\cX_F^{\gamma}=\cX_E^{\gamma}\cap\cX_F=\cB_E(\tau\cA_0,2\nu(\gamma))\cap \cX_F \subset \cB_F(v,\nu(\gamma)) = \cB_F(e,\nu(\gamma)-1/2),
\]
with equality in the tamely ramified case.

The proof is complete in all cases.
\end{proof}

\subsection{Proof of Proposition~\ref{proposition:orbital_integral_in_nonarch}}
In this proof, by a slight abuse of notation, we shall denote by $[x,y]$ the unique \emph{directed path} in $\cX_F$ that starts at the vertex $x$ and ends at the vertex $y$. That is, $[x,y]=[x',y']$ if and only if $x=x'$ and $y=y'$.

\subsubsection{The case of $\gamma$ split} By Lemma~\ref{lem1_nona}, we may assume that $\gamma\in\Kr$ is diagonal. Then, $G_{\gamma}$ consists of all diagonal matrices within $G$, the elements of $\Lambda_\gamma$ are the matrices $\diag(\varpi^n,\varpi^{-n})$ for $n\in\ZZ$, and the apartment stabilized by $\gamma$ is $\cA_0$.

First we consider the case $j=0$. By Lemma~\ref{lemma:subgroups_characterized_via_fixed_points}, we can think of the congruence subgroup
$K_0(\mpr^r)$ geometrically as $G_{[v_{-r},v_0]}$. Hence, applying Lemma~\ref{lemma:orbital_integral_to_counting} to
\[U=V=K_0(\mpr^r)=G_{[v_{-r},v_0]},\]
we obtain
\begin{equation}\label{from-lemma5-split-j0}
\bO(\gamma,\bchar_{K_0(\mpr^r)}) = \vol(K_0(\mpr^r)) \cdot \#\{g\in\Lambda_\gamma\bs G/G_{[v_{-r},v_0]}:\gamma gG_{[v_{-r},v_0]}=gG_{[v_{-r},v_0]}\}.
\end{equation}
The map $g\mapsto [gv_{-r},gv_0]$ gives rise to a bijection between $G/G_{[v_{-r},v_0]}$ and the orbit $G\cdot[v_{-r},v_0]$ of directed paths of length $r$ in $\cX_F$, with a compatible action of $G$ on these two sets. The condition $\gamma gG_{[v_{-r},v_0]}=gG_{[v_{-r},v_0]}$ is equivalent to $[g v_{-r},g v_0]\subset\cX_F^\gamma$, hence by Lemma~\ref{lemma:fixed_points},
\begin{equation}\label{count1}
\bO(\gamma,\bchar_{K_0(\mpr^r)}) \leq \vol(K_0(\mpr^r)) \cdot \#(\Lambda_\gamma \bs \{[x,y]\subset\cB_F(\cA_0,\nu(\gamma)):\dist_F(x,y)=r\}).
\end{equation}
This is the enumeration problem we have indicated earlier. For each directed path $[x,y]\subset\cB_F(\cA_0,\nu(\gamma))$, we define a vertex $v\in\cA_0$ as follows. If $[x,y]\cap\cA_0$ has at most one element, then $v$ is the point of $\cA_0$ closest to $[x,y]$. Otherwise $[x,y]\cap\cA_0$ is a path of positive length, and $v$ is its endpoint closer to $x$. The group $\Lambda_\gamma$ acts on $\cA_0$ by even shifts, hence we can assume that $v\in\{v_0,v_1\}$.

We consider two cases. Either $[x,y]\cap\cA_0$ has at most one element, or $[x,y]\cap\cA_0$ is a path of positive length. In the first case, let $p\in [x,y]$ be the vertex closest to $\cA_0$. As both $\dist_F(v,p)+\dist_F(p,x)$ and $\dist_F(v,p)+\dist_F(p,y)$ are at most $\nu(\gamma)$, the distance of $p$ from $v$ is at most
$\nu(\gamma)-r/2$; in particular, the first case is void when $r>2\nu(\gamma)$. Recording $s:=\min(\dist_F(p,x),\dist_F(p,y))$, for every $(r-\nu(\gamma))^+\leq s\leq\lfloor r/2\rfloor$ there are at most $6q^{\nu(\gamma)-(r-s)}$ choices for $p$, and for each $p$, there are at most
$2q^r$ choices for $[x,y]$. In total, there are at most $24q^{\nu(\gamma)+\lfloor r/2\rfloor}$ choices for $[x,y]$ in the first case.

In the second case, let $w$ be the endpoint of the path $[x,y]\cap \cA_0$ closer to $y$. We record the lengths $s:=\dist_F(x,v)$ and $t:=\dist_F(w,y)$; then, $[v,w]$ is a directed subpath of $[x,y]$ lying in $\cA_0$ of length $r-s-t\geq 1$. Considering first without loss of generality the case $s\leq t$, for every choice of $0\leq s\leq\min(\nu(\gamma),\lfloor\frac{r-1}2\rfloor)$ and $0\leq t\leq\min(r-1-s,\nu(\gamma))$, there are four choices for $[v,w]$, followed by $q^{s+t}$ choices for $[x,v]$ and $[w,y]$. In total we find for every $0\leq s\leq\min(\nu(\gamma),\lfloor\frac{r-1}2\rfloor)$ at most $8q^{\min(r-1,\nu(\gamma)+s)}$ choices for $[x,y]$. Summing over $s$ and taking into account the symmetry $s\leftrightarrow t$, we infer that there are at most
\[16\lfloor\tfrac{r+1}2\rfloor q^{r-1}\bchar_{1\leq r\leq\nu(\gamma)}+
16\left(2+\nu(\gamma)-\lceil\tfrac{r-1}2\rceil\right)q^{r-1}\bchar_{\nu(\gamma)+1\leq r\leq 2\nu(\gamma)}+
32q^{2\nu(\gamma)}\bchar_{r\geq 2\nu(\gamma)+1}\]
choices for $[x,y]$ in the second case. Altogether, the count on the right hand side of \eqref{count1} is
\[\ll q^{\nu(\gamma)+\lfloor r/2\rfloor}\bchar_{r\leq 2\nu(\gamma)}+
q^{r-1}\min(r,2\nu(\gamma)+1-r)\bchar_{1\leq r\leq 2\nu(\gamma)}+
q^{2\nu(\gamma)}\bchar_{r\geq 2\nu(\gamma)+1}
\ll q^{\nu(\gamma)+\min(\nu(\gamma),\lfloor r/2\rfloor)}.\]
This verifies \eqref{nonarchbound} in the case of $\gamma$ split and $j=0$, upon noting that ${|\Delta(\gamma)|}_F^{1/2}=q^{-\nu(\gamma)}$ and $E=F$.

Now we consider the case $j=1$. By Lemma~\ref{lemma:subgroups_characterized_via_fixed_points}, we can think of the congruence subgroup
$K_1(\mpr^r)$ geometrically as $G_{\cB_F(v_0,r)}$. Hence, applying Lemma~\ref{lemma:orbital_integral_to_counting} to
\[U=K_1(\mpr^r)=G_{\cB_F(v_0,r)}\qquad\text{and}\qquad V=K,\]
we obtain
\[ \bO(\gamma,\bchar_{K_1(\mpr^r)}) = \#\{g\in\Lambda_\gamma\bs G / K:\gamma gG_{\cB_F(v_0,r)}=gG_{\cB_F(v_0,r)}\}. \]
As in the case of $j=0$, the map $g\mapsto gv_0$ gives rise to a bijection between $G/K$ and the orbit $G\cdot v_0$ of vertices in $\cX_F$, with a compatible left action of $G$ on these two sets. The condition $\gamma gG_{\cB_F(v_0,r)}=gG_{\cB_F(v_0,r)}$ is equivalent to $\cB_F(gv_0,r)\subset\cX_F^\gamma$, hence by Lemma~\ref{lemma:fixed_points},
\begin{equation}\label{count3}
\bO(\gamma,\bchar_{K_1(\mpr^r)}) \leq \#(\Lambda_\gamma \bs \{\text{$x$ a vertex of $\cX_F$}:\cB_F(x,r)\subset\cB_F(\cA_0,\nu(\gamma))\}).
\end{equation}
As we divide by the action of $\Lambda_\gamma$, we may assume that the closest point of $\cA_0$ to $x$ is $v_0$ or $v_1$. Then there are at most $2 q^{\nu(\gamma)-r} \bchar_{r\leq\nu(\gamma)}$ choices for $x$. Noting also
\begin{equation}\label{indexbound}
\vol(K_1(\mpr^r))^{-1}=[K:K_1(\mpr^r)]<q^{3r}
\end{equation}
along with ${|\Delta(\gamma)|}_F^{1/2}=q^{-\nu(\gamma)}$ and $E=F$, we obtain \eqref{nonarchbound} in the case of $\gamma$ split and $j=1$.

If ${|\Delta(\gamma)|}_F=1$, then $\nu(\gamma)=0$, and we need to prove the more precise bound \eqref{nonarchbound2}. For $j=0$, similarly as in \eqref{from-lemma5-split-j0} and \eqref{count1}, we have to count those directed paths $[x,y]\subset\cA_0$ modulo $\Lambda_\gamma$ which belong to the $G$-orbit of $[v_{-r},v_0]$. The directed paths in question form the $\Lambda_\gamma$-orbit of $[v_{-r},v_0]$, hence
\[\bO(\gamma,\bchar_{K_0(\mpr^r)})=\vol(K_0(\mpr^r)).\]
For $j=1$, the condition $\cB_F(x,r)\subset\cA_0$ forces $r=0$, and then $K_0(\mpr^0)=K_1(\mpr^0)$ reduces the proof to the earlier discussed case $j=0$.

\subsubsection{The case of $\gamma$ elliptic and $E/F$ unramified}
From Lemma~\ref{lemma:fixed_points} we know that $\cX_F^\gamma$ equals $\cB_F(v,\nu(\gamma))$ for a suitable vertex $v\in\cX_F$, where also $\nu(\gamma)\in\NN$. Then, the analogue of \eqref{count1} reads (cf.\ Lemmata~\ref{lemma:orbital_integral_to_counting} and \ref{lemma:subgroups_characterized_via_fixed_points})
\begin{equation}\label{count2}
\bO(\gamma,\bchar_{K_0(\mpr^r)}) \leq \vol(K_0(\mpr^r)) \cdot \#\{[x,y]\subset\cB_F(v,\nu(\gamma)):\dist_F(x,y)=r\}.
\end{equation}
For any directed path $[x,y]\subset\cB_F(v,\nu(\gamma))$, let $p\in [x,y]$ be the vertex closest to $v$. As both $\dist_F(v,p)+\dist_F(p,x)$ and $\dist_F(v,p)+\dist_F(p,y)$ are at most $\nu(\gamma)$, the distance of $p$ from $v$ is at most $\nu(\gamma)-r/2$; so there is nothing to count when $r>2\nu(\gamma)$. Recording $s:=\min(\dist_F(p,x),\dist_F(p,y))$, for every $(r-\nu(\gamma))^+\leq s\leq\lfloor r/2\rfloor$
there are less than $3q^{\nu(\gamma)-(r-s)}$ choices for $p$, and for each $p$, there are at most $2q^r$ choices for $[x,y]$. Altogether, the count on the right hand side of \eqref{count2} is
\[\ll q^{\nu(\gamma)+\lfloor r/2\rfloor}\bchar_{r\leq 2\nu(\gamma)}.\]
This verifies \eqref{nonarchbound} in the case of $\gamma$ elliptic with unramified splitting field and $j=0$, upon noting that ${|\Delta(\gamma)|}_F^{1/2}=q^{-\nu(\gamma)}$ and
$\lambda=0$.

Similarly, the analogue of \eqref{count3} reads (cf.\ Lemmata~\ref{lemma:orbital_integral_to_counting} and \ref{lemma:subgroups_characterized_via_fixed_points})
\begin{equation}\label{count4}
\bO(\gamma,\bchar_{K_1(\mpr^r)}) \leq \#\{\text{$x$ a vertex of $\cX_F$}:\cB_F(x,r)\subset\cB_F(v,\nu(\gamma))\}.
\end{equation}
Clearly, there are less than $3 q^{\nu(\gamma)-r} \bchar_{r\leq\nu(\gamma)}$ choices for $x$. Noting also \eqref{indexbound}
along with ${|\Delta(\gamma)|}_F^{1/2}=q^{-\nu(\gamma)}$ and $\lambda=0$, we obtain \eqref{nonarchbound} in the case of $\gamma$ elliptic with unramified splitting field and $j=1$.

If ${|\Delta(\gamma)|}_F=1$, then $\nu(\gamma)=0$, and we need to prove the more precise bound \eqref{nonarchbound2}. In this case, both counts on the right hands sides of \eqref{count2} and \eqref{count4} are equal to $\bchar_{r=0}$, hence \eqref{nonarchbound2} follows.

\subsubsection{The case of $\gamma$ elliptic and $E/F$ ramified}
From Lemma~\ref{lemma:fixed_points} we know that $\cX_F^\gamma\subset\cB_F(e,\nu(\gamma)-1/2)$ for a suitable edge $e\subset\cX_F$, where also $2\nu(\gamma)-1\in\NN$. Hence the analogue of \eqref{count1} and \eqref{count2} reads (cf.\ Lemmata~\ref{lemma:orbital_integral_to_counting} and \ref{lemma:subgroups_characterized_via_fixed_points})
\begin{equation}\label{count5}
\bO(\gamma,\bchar_{K_0(\mpr^r)}) \leq \vol(K_0(\mpr^r)) \cdot \#\{[x,y]\subset\cB_F(e,\nu(\gamma)-1/2):\dist_F(x,y)=r\}.
\end{equation}
For estimating the number of directed paths $[x,y]\subset\cB_F(e,\nu(\gamma)-1/2)$ of length $r$, we distinguish between two cases.
If $[x,y]$ contains $e$, then we argue similarly as in the subcase of the split case when $[x,y]$ intersected $\cA_0$ in a path of positive length; in fact, the present case is a bit easier since the lengths of the (up to two) components of $[x,y]\setminus e$ determine each other. Otherwise, we argue as in the unramified elliptic case, with $\lfloor\nu(\gamma)-1/2\rfloor$ in place of $\nu(\gamma)$. Altogether, the count on the right hand side of \eqref{count5} is
\[\ll q^{r-1}\min(r,2\nu(\gamma)+1-r)\bchar_{1\leq r\leq 2\nu(\gamma)}+q^{\lfloor\nu(\gamma)-1/2\rfloor+\lfloor r/2\rfloor}\bchar_{r\leq 2\nu(\gamma)-1}\ll q^{\nu(\gamma)-1/2+\lfloor r/2\rfloor}\bchar_{r\leq 2\nu(\gamma)}.\]
This verifies \eqref{nonarchbound} in the case of $\gamma$ elliptic with ramified splitting field and $j=0$, upon noting that ${|\Delta(\gamma)|}_F^{1/2}=q^{-\nu(\gamma)}$ and $\lambda=1$.

Similarly, the analogue of \eqref{count3} and \eqref{count4} reads (cf.\ Lemmata~\ref{lemma:orbital_integral_to_counting} and \ref{lemma:subgroups_characterized_via_fixed_points})
\[\bO(\gamma,\bchar_{K_1(\mpr^r)}) \leq \#\{\text{$x$ a vertex of $\cX_F$}:\cB_F(x,r)\subset\cB_F(e,\nu(\gamma)-1/2)\}.\]
Clearly, there are at most $4 q^{\nu(\gamma)-1/2-r} \bchar_{r\leq\nu(\gamma)-1/2}$ choices for $x$. Noting also \eqref{indexbound}
along with ${|\Delta(\gamma)|}_F^{1/2}=q^{-\nu(\gamma)}$ and $\lambda=1$, we obtain \eqref{nonarchbound} in the case of $\gamma$ elliptic with ramified splitting field and $j=1$.

\bigskip

The proof of Proposition~\ref{proposition:orbital_integral_in_nonarch} is complete in all cases.

\section{Global aspects}\label{sec:Global}
We collect some notations important for the following two sections. Let $a,b,c\in\NN$ with $a+b\geq 1$, and recall that all implied constants are allowed to depend on these parameters. Let $k$ be a number field with $a+c$ real and $b$ complex places. Let $A$ be a division quaternion algebra over $k$ of signature $(a,b,c)$. Let
\[\ram(A)=\ram_\infty(A)\cup\ram_f(A)\]
be the set of places (resp. archimedean places and non-archimedean places) where $A$ ramifies. Let $\bG:=\SL_1(A)$ and $\bH:=\GL_1(A)$ viewed as algebraic groups over $k$, and let $G:=\SL_2(\RR)^a\times\SL_2(\CC)^b$. As $\ram(A)$ is non-empty, every regular semisimple element $\gamma\in\bGr(k)$ has a quadratic splitting field $k(\gamma)$ over $k$. In other words, the centralizer $\bG_\gamma$ is an anisotropic algebraic torus of rank one defined over $k$.

According to \eqref{GA-decomp} and \eqref{eq:Ident1}, we can identify
\begin{equation}\label{GA-product}
\bG(\AA)=\SL_2(\RR)^a\times\SL_2(\CC)^b\times\SU_2(\CC)^c\times\prod_{\mpr\in\ram_f(A)}\SL_1(D_\mpr)\times\prod_{\mpr\not\in\ram_f(A)}\SL_2(k_{\mpr}).
\end{equation}
We specify the \emph{standard measure} $\mu$ on $\bG(\AA)$ (resp. $\bG_\gamma(\AA)$). Let $v$ be a place of $k$. For $v\in\ram(A)$, let $\mu_v$ be the Haar probability measure on $\bG(k_v)$. For $v\not\in\ram(A)$ archimedean (resp. non-archimedean), let $\mu_v$ be the Haar measure on $\bG(k_v)$ defined in $\S\ref{sec:HaarArch}$ (resp. \S\ref{sec:HaarNarch}). We define the global measure $\mu=\mu_\infty\times\mu_f$ on $\bG(\AA)=\bG(\AA_\infty)\times\bG(\AA_f)$ as the product of the local measures $\mu_v$. By an abuse of notation, we also denote by $\mu_v$ the Haar measure on $\bG_\gamma(k_v)$ defined in \S\ref{sec:ArchimedeanOI} and \S\ref{sec:HaarNarch}, and we define $\mu:=\prod_v\mu_v$ on $\bG_\gamma(\AA)$. Finally, we endow the lattices $\bG(k)$ and $\bG_{\gamma}(k)$ with counting measures and the corresponding quotients $\bG(k)\bs\bG(\AA)$ and $\bG_{\gamma}(k)\bs\bG_{\gamma}(\AA)$ with the unique compatible right $\bG(\AA)$-invariant (resp. right $\bG_{\gamma}(\AA)$-invariant) measures.

\subsection{Borel's volume formula}\label{sec:Borel}
In this section, we record the classical covolume of the group of norm $1$ elements of a maximal order of $A$. The result is really due to Borel~\cite[\S 7.3]{B}, but we differ from his formula by a factor of $2$ as we work with special linear groups rather than projective ones. For clarity, we provide a detailed proof based on the local calculations in \cite[pp.~17--18]{B} and the fact that the \emph{Tamagawa number} of $\bG$ equals $1$ \cite[Th.~3.3.1]{W2}. As a supplement, we also provide an explicit lower bound for the covolume, based on Odlyzko's discriminant bound \cite[Th.~1]{Od}.

\begin{proposition}\label{prop:Borel} Let $\Gamma:=\Gamma_\emptyset(\mo)$ in the notation of \eqref{Gamma-def}. Then
\[\vol(\Gamma\bs G)=\frac{\zeta_k(2)\Delta_k^{3/2}}{2^{3b+2c}\pi^{a+2b+2c}}\prod_{\mpr\in\ram_f(A)}(N(\mpr)-1).\]
Moreover, $\vol(\Gamma\bs G)>e^{-7}$.
\end{proposition}

\begin{proof} Let us consider, in the notation of \eqref{K-def},
\[V:=\SU_2(\CC)^c\times K_\emptyset(\mo).\]
Then $\bG(\AA)=\bG(k)(G\times V)$ by the strong approximation property \cite[Th.~7.7.5]{MaRe03}, and hence
\[\bG(k)\bs\bG(\AA)/V\simeq\eta(\bG(k)\cap(G\times V))\bs G=\Gamma\bs G,\]
where $\eta\colon\bG(\AA)\to G$ is the projection introduced below \eqref{GA-decomp}. As $V$ is a product of maximal compact subgroups over the relevant places of $k$, the definition of the standard measure yields
\begin{equation}
\label{mu-to-vol}
\mu(\bG(k)\bs\bG(\AA))=\vol(\Gamma\bs G).
\end{equation}
It remains to calculate the left hand side. The standard measure $\mu$ on $\bG(\AA)$ is proportional to the \emph{Tamagawa measure} $\tau=\Delta_k^{-3/2}\omega$, as defined in \cite[\S 6.2]{B}. Combining the two Lemmata on \cite[pp.~17--18]{B} with (3) on \cite[p.~18]{B}, we see that $\mu=c\,\omega$, where
\[c^{-1}=\pi^a\left(8\pi^2\right)^b\left(4\pi^2\right)^c\
\prod_{\mpr\in\ram_f(A)}\frac{N(\mpr)+1}{N(\mpr)^2}\ \prod_{\mpr\not\in\ram_f(A)}\frac{N(\mpr)^2-1}{N(\mpr)^2}.\]
On the other hand, $\tau(\bG(k)\bs\bG(\AA))=1$ by \cite[Th.~3.3.1]{W2}, therefore
\[\mu(\bG(k)\bs\bG(\AA))=c\,\Delta_k^{3/2}=\frac{\zeta_k(2)\Delta_k^{3/2}}{2^{3b+2c}\pi^{a+2b+2c}}\prod_{\mpr\in\ram_f(A)}(N(\mpr)-1).\]

Finally, we show that the right hand side exceeds $e^{-7}$. For this we apply \cite[Th.~1]{Od} with the parameters $\sigma:=3.399$ and $\tilde\sigma:=2.762$. The conditions \cite[(1.5a)--(1.5b)]{Od} are satisfied, hence \cite[(1.6)]{Od} yields the explicit bound
\[\Delta_k>11.5948^{a+c}\ 7.6189^{2b}\ e^{-4.0497}.\]
We obtained the numerical values with SageMath~8.3, and verified them independently with Mathematica~10.1. Therefore,
\[\frac{\Delta_k^{3/2}}{2^{3b+2c}\pi^{a+2b+2c}}>
\left(\frac{39.48}{4\pi^2}\right)^{a+c}\left(\frac{442.25}{8\pi^2}\right)^b e^{-6.0746}.\]
The fractions on the right hand side exceed $1$, hence we are done.
\end{proof}

\subsection{Volumes of centralizers}\label{sec:Volumes}
In this section, we estimate $\mu(\bG_\gamma(k)\bs\bG_\gamma(\AA))$ for $\gamma\in\bGr(k)$. Before stating and proving the actual result, we collect some relevant facts about algebraic tori.

Let $\bG_m$ be the multiplicative group, regarded as a torus of rank one defined over $\QQ$. Let $k$ be a number field, and let $\bT$ be a torus of rank $r$ defined over $k$. The character group
\[X^*(\bT):=\Hom(\bT_{\ov{\QQ}},\bG_{m\ov{\QQ}})\]
is a free abelian group of rank $r$. A Galois automorphism $\sigma\in\Gal(\QQ)$ determines the torus $\sigma\bT$ of rank $r$ defined over $\sigma k$, and an isomorphism of abelian groups
\[\sigma\colon X^*(\bT)\xrightarrow{\sim} X^*(\sigma\bT).\]
The torus $\sigma\bT$ only depends on the coset $\sigma\Gal(k)$, hence in particular $X^*(\bT)\simeq\ZZ^r$ is a $\Gal(k)$-module. Similarly, $\bT':=\Res_{k/\QQ}\bT$ is a torus of rank $r[k:\QQ]$ defined over $\QQ$, hence $X^*(\bT')\simeq\ZZ^{r[k:\QQ]}$ is a $\Gal(\QQ)$-module.
In fact \cite[\S 2.61]{Mi} shows that
\[\bT'_{\ov{\QQ}}\simeq\prod_{\sigma\in\Gal(\QQ)/\Gal(k)}(\sigma\bT)_{\ov{\QQ}},\]
therefore
\[X^*(\bT')\simeq\prod_{\sigma\in\Gal(\QQ)/\Gal(k)}\sigma X^*(\bT)\simeq\Ind_{\Gal(k)}^{\Gal(\QQ)}X^*(\bT).\]

For an arbitrary place $w$ of $\QQ$, there is an isomorphism of $k$-algebras
\[\QQ_w\otimes_\QQ k\simeq\prod_{v\mid w}k_v,\]
where $v$ runs through the places of $k$ lying above $w$. Correspondingly, there is an isomorphism of topological groups
\[\bT'(\QQ_w)\simeq\prod_{v\mid w}\bT(k_v).\]
Let us identify $\Gal(\QQ_w)$ with a subgroup of $\Gal(\QQ)$ in the usual way. In the left action of $\Gal(\QQ)$ on $\Gal(\QQ)/\Gal(k)$, each orbit of $\Gal(\QQ_w)$ corresponds to a place $v$ of $k$ lying above $w$, and the orbit itself can be identified with $\Gal(\QQ_w)/\Gal(k_v)$. It follows that
\[X^*(\bT')^{\Gal(\QQ_w)}\simeq\prod_{v\mid w}X^*(\bT)^{\Gal(k_v)}.\]
This isomorphism commutes with evaluation of local characters, so the left hand side as a subgroup of $\Hom(\bT'(\QQ_w),\QQ_w^\times)$ corresponds to the right hand side as a direct product of subgroups of $\Hom(\bT(k_v),k_v^\times)$. This will show that the \emph{standard measure} on $\bT'(\QQ_w)$, to be defined in the next paragraph, equals the product of the analogous standard measures on $\bT(k_v)$.

Following \cite[\S 2.1]{Ono}, the topological group $\bT'(\QQ_w)$ has unique maximal compact subgroup
\[\bT'(\QQ_w)^\flat:=\{t\in\bT'(\QQ_w): \text{${|\xi(t)|}_w=1$ for all $\xi\in X^*(\bT')^{\Gal(\QQ_w)}$}\}.\]
By choosing a(ny) basis $\{\xi_1,\ldots, \xi_s\}$ of $X^*(\bT')^{\Gal(\QQ_w)}$, we see that $\bT'(\QQ_w)/\bT'(\QQ_w)^\flat$ is isomorphic to $\ZZ^s$ when $w$ is non-archimedean (resp. $\RR^s$ when $w$ is archimedean). Accordingly, we define the standard measure $\mu_w$ on $\bT'(\QQ_w)$ as the product of the Haar probability measure $\mu_w^\flat$ on $\bT'(\QQ_w)^\flat$ and the counting measure on $\ZZ^s$ (resp. Lebesgue measure on $\RR^s$). Explicitly, when $w$ is non-archimedean, $\mu_w$ is the unique Haar measure on $\bT'(\QQ_w)$ such that $\mu_w(\bT'(\QQ_w)^\flat)=1$; when $w$ is archimedean,
\[\int_{\bT'(\QQ_w)} f(t)\,d\mu_w(t)=
\int_{\bT'(\QQ_w)/\bT'(\QQ_w)^\flat}\left(\int_{\bT'(\QQ_w)^\flat}f(t't'')\,d\mu_w^\flat(t'')\right)
\frac{d{|\xi_1(t')|}_w}{{|\xi_1(t')|}_w}\cdots \frac{d{|\xi_s(t')|}_w}{{|\xi_s(t')|}_w}.\]

The \emph{standard measure} $\mu$ on $\bT'(\AA_\QQ)$ is the product of the local measures $\mu_w$. The standard measure $\mu$ on $\bT(\AA_k)$ is defined analogously, using the characters of $\bT$ defined over the various completions $k_v$. By our earlier remarks, these two measures correspond to each other under the natural isomorphism $\bT'(\AA_\QQ)\simeq\bT(\AA_k)$.

\begin{proposition}\label{prop:TorusVolume}
Let $\gamma\in\bGr(k)$ be a regular semisimple element. Then
\[\mu(\bG_\gamma(k)\bs\bG_\gamma(\AA))\ll_\eps\Delta_k^{1/2+\eps}|N_{k/\QQ}(\Delta_{k(\gamma)/k})|^{1/2+\eps}.\]
\end{proposition}

\begin{proof} Let us abbreviate $\bT:=\bG_\gamma$ and $l:=k(\gamma)$. By our initial remarks,
\begin{equation}\label{measuresame}
\mu(\bT(k)\bs\bT(\AA_k))=\mu(\bT'(\QQ)\bs\bT'(\AA_\QQ)),
\end{equation}
where $\bT':=\Res_{k/\QQ}\bT$ is an anisotropic torus of rank $[k:\QQ]$ defined over $\QQ$. The standard measure $\mu$ on $\bT'(\AA)$ is proportional to the \emph{Tamagawa measure} $\omega$ on $\bT'(\AA_\QQ)$, as defined by \cite[(12)]{UY}. The square of the proportionality constant is called the \emph{quasi-discriminant} $D_{\bT'}$ \cite[(13)]{UY}, that is, $\mu=D_{\bT'}^{1/2}\omega$. Moreover, by the definition of the \emph{Tamagawa number} $\tau_{\bT'}$ \cite[p.~126]{Ono},
\[\omega(\bT'(\QQ)\bs\bT'(\AA_\QQ))=\tau_{\bT'}\,\rho_{\bT'},\]
where $\rho_{\bT'}$ is the special value at $s=1$ of the Artin $L$-function
\[L(s,\bT'):=L(s,X^*(\bT')\otimes\CC).\]
It follows that
\[\mu(\bT'(\QQ)\bs\bT'(\AA_\QQ))=\tau_{\bT'}\,\rho_{\bT'}\,D_{\bT'}^{1/2}.\]
By \cite[Prop.~3.1]{UY}, the quasi-discriminant $D_{\bT'}$ is smaller than the conductor $a_{\bT'}$ of $L(s,\bT')$.
On the other hand, by our initial remarks, $L(s,\bT')$ is the same as $L(s,\bT)$, hence $a_{\bT'}=a_{\bT}$ and $\rho_{\bT'}=\rho_{\bT}$. In addition, $\tau_{\bT'}=\tau_{\bT}$ by \cite[Th.~3.5.1]{Ono}. Recalling also \eqref{measuresame}, we conclude that
\[\mu(\bT(k)\bs\bT(\AA_k))<\tau_{\bT}\,\rho_{\bT}\,a_{\bT}^{1/2}.\]

The absolute Galois group $\Gal(k)$ acts on $X^*(\bT)\simeq\ZZ$ by the nontrivial quadratic character $\chi$ fixing $\Gal(l)$. Therefore,
by the Corollary on \cite[p.~69]{Ono2} and an explicit calculation with cocycles and coboundaries (cf.\ \cite[p.~97]{CF}),
\[\tau_{\bT}=\#H^1(\Gal(l/k),X^*(\bT))=2.\]
Moreover,
\[L(s,\bT)=L(s,\chi)=\zeta_l(s)/\zeta_k(s),\]
whose conductor equals
\[a_{\bT}=\Delta_l/\Delta_k=\Delta_k\,|N_{k/\QQ}(\Delta_{l/k})|.\]
Finally, by Lemma~4 in \cite[Ch.~XVI, \S 3]{Lang},
\[\rho_{\bT}=L(1,\chi)\ll_\eps a_{\bT}^\eps.\]

Putting everything together,
\[\mu(\bT(k)\bs\bT(\AA_k))\ll_\eps a_{\bT}^{1/2+\eps}=\Delta_k^{1/2+\eps}|N_{k/\QQ}(\Delta_{l/k})|^{1/2+\eps},\]
and we are done.
\end{proof}

\subsection{Counting conjugacy classes}\label{sec:ConjClasses}
Let us introduce the notation
\[k_A^+:=\{x\in k^\times:\text{$x_v>0$ for all $v\in\ram_\infty(A)$}\}.\]

\begin{lemma}\label{lemma11} Let $\gamma\in\bGr(k)$ be a regular semisimple element. There is a bijection
\[\{[h^{-1}\gamma h]\colon h\in\bH(k)\}\xrightarrow{\sim}k_A^+/N_{k(\gamma)/k}(k(\gamma)^\times).\]
Explicitly, the $\bG(k)$-conjugacy class $[h^{-1}\gamma h]$ is mapped to the image of $\nn(h)$ under the natural quotient map.
\end{lemma}

\begin{proof} For arbitrary $h_1,h_2\in\bH(k)$, we can see that
\begin{align*}
[h_1^{-1}\gamma h_1]=[h_2^{-1}\gamma h_2]
&\qquad\Longleftrightarrow\qquad\exists g\in\bG(k):h_1^{-1}\gamma h_1 =g^{-1}h_2^{-1}\gamma h_2 g\\
&\qquad\Longleftrightarrow\qquad\exists g\in\bG(k):h_1 h_2^{-1}\in\bH_\gamma(k) h_2 g h_2^{-1}\\
&\qquad\Longleftrightarrow\qquad h_1h_2^{-1}\in\bH_\gamma(k)\bG(k).
\end{align*}
So we can identify
\[\{[h^{-1}\gamma h]:h\in\bH(k)\}\simeq\bH(k)/(\bH_\gamma(k)\bG(k))\simeq\nn(\bH(k))/\nn(\bH_\gamma(k)),\]
where the second bijection is a group isomorphism induced by the reduced norm. Let us abbreviate $l:=k(\gamma)$ and write $\sigma$ for the nontrivial Galois automorphism of $l$ over $k$. Then $\nn(\bH(k))=k_A^+$ by the Hasse--Schilling theorem \cite[Ch.~XI, Prop.~3]{W}, while $\nn(\bH_\gamma(k))=N_{l/k}(l^\times)$, because $\bH_\gamma(k)$ is conjugate to $\{\diag(\beta,\beta^\sigma):\beta\in l^\times\}$ inside $\bH(l)\simeq\GL_2(l)$. The result follows.
\end{proof}

We recall the notation \eqref{K-def} and the decomposition \eqref{GA-product}. Let us also introduce
\[\cK:=\bigcup_{g\in\bG(\AA_f)}g K_\emptyset(\mo)g^{-1}.\]
Using Lemma~\ref{lemma11}, we can bound effectively the number of $\bG(k)$-conjugacy classes which intersect $\cK$ and lie inside a given $\bH(k)$-conjugacy class.

\begin{lemma}\label{lem:ConjCount}
Let $\gamma\in\bGr(k)\cap\cK$ be a regular semisimple element. The cardinality of
\[[\gamma]':=\{[h^{-1}\gamma h]:\text{$h\in\bH(k)$ and $h^{-1}\gamma h\in\cK$}\}\]
is at most $2^{a+|\!\ram_f(A)|+\omega(\Delta(\gamma))}$, where $\omega(\Delta(\gamma))$ is the number of distinct prime ideals dividing $\Delta(\gamma)$.
\end{lemma}

\begin{proof} Let us abbreviate $l:=k(\gamma)$. By Lemma~\ref{lemma11}, we can think of $[\gamma]'$ as a subset of $k_A^+/N_{l/k}(l^\times)$. On the other hand, by the Hasse norm theorem \cite[Ch.~VI, Cor.~4.5]{Ne}, there is a natural embedding of groups
\[k_A^+/N_{l/k}(l^\times)\hookrightarrow\bigoplus_{v\not\in\ram_\infty(A)}k_v^\times / N_{l_v/k_v}(l_v^\times),\]
where $l_v:=l\otimes_k k_v$. The index $[k_v^\times : N_{l_v/k_v}(l_v^\times)]$ equals $2$ or $1$ depending on whether $l_v$ is a field or not. Hence it suffices to verify that $\nn(h)\in N_{l_\mpr/k_\mpr}(l_\mpr^\times)$ always holds when:
\begin{itemize}
\listsep
\item $\mpr$ is a fixed prime such that $\mpr\not\in\ram_f(A)$, $\mpr\nmid\Delta(\gamma)$, and $l_\mpr$ is a field;
\item $h\in\GL_2(k_\mpr)$ and $g\in\SL_2(k_\mpr)$ are such that $g^{-1}h^{-1}\gamma hg\in\SL_2(\mo_\mpr)$.
\end{itemize}
By the initial assumptions, there exist $i\in\SL_2(k_\mpr)$ and $\delta\in\SL_2(\mo_\mpr)$ such that $\gamma=i^{-1}\delta i$. Of course $\Delta(\gamma)=\Delta(\delta)$. The conditions imply that, in the Bruhat--Tits tree $\cX_{k_\mpr}$, the vertices $v_0$ and $(ihg)v_0$ are fixed by the regular semisimple element $\delta\in\SL_2(\mo_\mpr)$. From Lemma~\ref{lemma:fixed_points} we infer that the quadratic extension $l_\mpr/k_\mpr$
is unramified (which also follows more directly from $\mpr\nmid\Delta(\gamma)$), and there is \emph{exactly one} vertex fixed by $\delta$. We
conclude that $v_0=(ihg)v_0$, hence $v_\mpr(\nn(ihg))\in 2\ZZ$, and so $\nn(h)=\nn(ihg)\in N_{l_\mpr/k_\mpr}(l_\mpr^\times)$ by \cite[Ch.~VIII, Prop.~3]{W}. The proof is complete.
\end{proof}

The next two simple estimates are similar in nature, and they will be used in conjunction with Lemma~\ref{lem:ConjCount}.

\begin{lemma}\label{twopower1} Let $\Gamma:=\Gamma_\emptyset(\mo)$ in the notation of \eqref{Gamma-def}. Then
\[2^{|\!\ram_f(A)|}\ll_\eps\vol(\Gamma\bs G)^\eps.\]
\end{lemma}

\begin{proof} Let $\eps>0$ be arbitrary. By Proposition~\ref{prop:Borel},
\[\prod_{\mpr\in\ram_f(A)}2\ll_\eps\prod_{\substack{\mpr\in\ram_f(A)\\N(\mpr)\geq 1+2^{1/\eps}}}2
\leq\prod_{\mpr\in\ram_f(A)}(N(\mpr)-1)^\eps\ll_\eps\vol(\Gamma\bs G)^\eps.\]
The proof is complete.
\end{proof}

\begin{lemma}\label{twopower2}
Let $\fm\subset\mo$ be a nonzero ideal, and let $\omega(\fm)$ be the number of its prime ideal divisors. Then
\[2^{\omega(\fm)}\ll_\eps N(\fm)^\eps.\]
\end{lemma}

\begin{proof} Let $\eps>0$ be arbitrary. Then
\[\prod_{\mpr\mid\fm}2\ll_\eps\prod_{\substack{\mpr\mid\fm\\N(\mpr)\geq 2^{1/\eps}}}2
\leq\prod_{\mpr\mid\fm}N(\mpr)^\eps\leq N(\fm)^\eps.\]
The proof is complete.
\end{proof}

\section{Proof of Theorem~\ref{thm:Main}}\label{sec:Proof}
In this section, we prove Theorem~\ref{thm:Main}. Recall the parameters $S\subset S_\infty^G$, $\bm\sigma=(\sigma_j)\in [0,1/2]^S$, and $\bm T=(T_j)\in \RR^{S^G_\infty\setminus S}$.

In \S\ref{tf-sec}, we construct a positive definite test function $f_\AA\in C_c(\bG(\AA))$. In \S\ref{spectralside-sec}, we show that the spectral side of the trace formula for the associated integral operator $Rf_\AA$ detects with high weights the multiplicities $\mm(\pi,\Gamma_\kappa(\mn))$ for $\pi\in\cB(\bm \sigma, \bm T)$. In \S\ref{geometricside-sec}, we estimate the contributions of individual conjugacy classes to the geometric side of the trace formula for $Rf_\AA$. In \S\ref{geometryofnumbers-sec}, we combine the bounds from the spectral and geometric expansions with a crucial geometry of numbers ingredient from \cite{FHM20} to conclude the proof of Theorem~\ref{thm:Main}.

\subsection{Test function}\label{tf-sec}
Our construction of the test function $f_\AA\in C_c(\bG(\AA))$ will depend on an additional family of nonnegative parameters $\bm R=(R_j)\in\RR_{\geq 0}^S$, to be specified later in the proof of Theorem~\ref{thm:Main}.

For $j\in S_{\infty}^G$, let $(G_j,K_j)$ be the pair $(\SL_2(\RR),\SO_2(\RR))$ for $j\in\{1,\dotsc,a\}$
and the pair $(\SL_2(\CC),\SU_2(\CC))$ for $j\in\{a+1,\dotsc,a+b\}$. Then, $G=\prod_{j=1}^{a+b}G_j$ is the Lie group in Theorem~\ref{thm:Main}, and $K:=\prod_{j=1}^{a+b}K_j$ is a maximal compact subgroup as in \S\ref{lattices-subsubsection}. For $j\in S$, let $f_j\in C_c(K_j\bs G_j/K_j)$ be the spherical function afforded by Proposition~\ref{nontempered} for $R=R_j$. For $j\in S_\infty^G\setminus S$, let $f_j\in C_c(K_j\bs G_j/K_j)$ be the spherical function afforded by Proposition~\ref{tempered} for $t=T_j$. We define
\[f:=\mathop{\textstyle{\bigotimes}}\nolimits_{j=1}^{a+b} f_j\in C_c(K\bs G/K),\]
as well as
\begin{equation}\label{fAA-def}
f_\AA:=f \otimes\bchar_{V}/\mu_f(K_\kappa(\mn)),\qquad V:=\SU_2(\CC)^c\times K_\kappa(\mn),
\end{equation}
where $\mu_f$ is the standard measure on $\bG(\AA_f)$ from the beginning of Section~\ref{sec:Global}.

By Propositions~\ref{nontempered} and \ref{tempered}, the function $f_\AA\in C_c(\bG(\AA))$ is supported in $\cD(\bm R)\times V$, where
\begin{equation}\label{DR-def}
\cD(\bm R):=\prod_{j\in S} B(2R_j+2)\,\times\prod_{j\in S_\infty^G\setminus S} B(2),
\end{equation}
with $B(R)$ as in \eqref{BRdef} and with the obvious ordering of the factors. Moreover,
\begin{equation}\label{fAA-bound}
f_\AA(g)\ll\mu_f(K_\kappa(\mn))^{-1}\prod_{j\in S_{\infty}^G\setminus S}(1+|T_j|)^{\rho_j},\qquad g\in\bG(\AA).
\end{equation}

\subsection{Spectral side of the trace formula}\label{spectralside-sec}
Let $Rf_\AA\colon L^2(\bG(k)\bs\bG(\AA))\to L^2(\bG(k)\bs\bG(\AA))$ be the positive trace class operator given by
\[Rf_\AA \phi(h):=\int_{\bG(\AA)}\phi(hg)f_\AA(g)\,dg.\]
The image of this operator is contained in the space $L^2(\bG(k)\bs\bG(\AA))^{K\times V}$ of vectors fixed by
\[K\times V=\SO_2(\RR)^a\times\SU_2(\CC)^b\times\SU_2(\CC)^c\times K_\kappa(\mn),\] and so
\begin{equation}\label{restrictedtrace}
\tr Rf_\AA=\tr Rf_\AA|_{L^2(\bG(k)\bs\bG(\AA))^{K\times V}}.
\end{equation}

We can identify $L^2(\bG(k)\bs\bG(\AA))^{K\times V}$ with a classical function space on $\Gamma_{\kappa}(\mn)\bs G/K$ as in \eqref{identification-sum} below. Indeed, $\bG(\AA)=\bG(k)(G\times V)$ by the strong approximation property \cite[Th.~7.7.5]{MaRe03}, and hence
\[\bG(k)\bs\bG(\AA)/V\simeq\eta(\bG(k)\cap(G\times V))\bs G=\Gamma_\kappa(\mn)\bs G,\]
where $\eta\colon\bG(\AA)\to G$ is the projection introduced below \eqref{GA-decomp}. It follows that $L^2(\bG(k)\bs\bG(\AA))^V$ is isomorphic to $L^2(\Gamma_\kappa(\mn)\bs G)$, and
\begin{equation}\label{identification-sum}
L^2(\bG(k)\bs\bG(\AA))^{K\times V}\simeq L^2(\Gamma_\kappa(\mn)\bs G)^{K}\simeq
\wbigoplus_{\bm{s}}\mm(\pi_{\bm{s}},\Gamma_\kappa(\mn))\pi_{\bm{s}}^{K},
\end{equation}
recalling the parametrization of $\pi_{\bm{s}}\in\wG$ by $\bm{s}\in((0,1/2]\cup i[0,\infty))^{a+b}$ from \S\ref{reps-subsubsection}.
The action of $Rf_\AA$ on the left hand side becomes the action of the classical operator $Rf$ (cf.\ \eqref{R-gamma-def}) on the space in the middle.
From \eqref{spherical}, it follows that this action becomes multiplication by the scalar
\[\widehat{f}(\bm{s}):=\prod_{j=1}^{a+b}\widehat{f_j}(s_j)\]
on each of the one-dimensional Hilbert space constituents $\pi_{\bm{s}}^{K}\simeq\bigotimes_{j=1}^{a+b}\pi_{s_j}^{K_j}$ on the right hand side. Using also \eqref{restrictedtrace}, we arrive at
\begin{equation}\label{traceviaspherical}
\tr Rf_\AA=\sum_{\bm{s}}\mm(\pi_{\bm{s}},\Gamma_\kappa(\mn))\widehat{f}(\bm{s}).
\end{equation}
The right hand side of \eqref{traceviaspherical} expresses the spectral side of the trace formula for $f_{\AA}$.

For each $\pi_{\bm{s}}$ that occurs in \eqref{traceviaspherical}, we have $\widehat{f}(\bm{s})\geq 0$ by the positivity of $Rf_\AA$. Moreover, Propositions~\ref{nontempered} and \ref{tempered} yield
\[\widehat{f}(\bm{s})\gg\prod_{j\in S}e^{2\rho_j R_j\sigma_j},\qquad\pi_{\bm{s}}\in\cB(\bm\sigma,\bm T).\]
Therefore,
\begin{equation}\label{multiplicityviatrace}
\sum_{\pi\in\cB(\bm{\sigma},\bm{T})}\mm(\pi,\Gamma_\kappa(\mn))\ll e^{-\sum_{j\in S}2\rho_j R_j\sigma_j}\tr Rf_\AA.
\end{equation}

\subsection{Geometric side of the trace formula}\label{geometricside-sec}
By the trace formula for compact quotients \cite[Part~I, \S 1]{A}, we also have the geometric expansion
\begin{equation}\label{eq:geom_exp}
\tr Rf_\AA=\sum_{[\gamma]\subset\bG(k)}\mu(\bG_\gamma(k)\bs\bG_\gamma(\AA))\bO(\gamma,f_\AA),
\end{equation}
where $[\gamma]$ runs through the conjugacy classes of $\bG(k)$, $\mu$ is the standard measure on $\bG_\gamma(\AA)$ defined in the beginning of Section~\ref{sec:Global}, and
\[\bO(\gamma,f_\AA):=\int_{\bG_\gamma(\AA)\bs\bG(\AA)} f_\AA(g^{-1}\gamma g)\,dg\]
is the global orbital integral. Separating the contribution of the central and regular semisimple conjugacy classes, we obtain
\[\tr Rf_\AA=\mu(\bG(k)\bs\bG(\AA))(f_\AA(\id)+f_\AA(-\id))+\sum_{[\gamma]\in W}\mu(\bG_\gamma(k)\bs\bG_\gamma(\AA))\bO(\gamma,f_\AA),\]
where $W:=W(\bm{R},K_\kappa(\mn))$ is the set of conjugacy classes $[\gamma]\subset\bGr(k)$ such that the conjugacy class of $\gamma$ in $\bG(\AA)$ intersects $\cD(\bm{R})\times V$ (cf.\ \eqref{fAA-def}--\eqref{DR-def}). Using also \eqref{fAA-bound}, \eqref{mu-to-vol} for $\Gamma=\Gamma_\emptyset(\mo)$, and \eqref{conductors} for $\Gamma=\Gamma_\kappa(\mn)$, we infer that
\begin{equation}\label{tracebound1}
\tr Rf_\AA\ll\cC(\Gamma_\kappa(\mn),\bm{T})+\Biggl|\sum_{[\gamma]\in W}\mu(\bG_\gamma(k)\bs\bG_\gamma(\AA))\bO(\gamma,f_\AA)\Biggr|.
\end{equation}
Note that $\tr(\gamma)\in\mo$ holds for every $[\gamma]\in W$, because $\gamma$ lies in a maximal compact subgroup of $\bG(\AA_f)$.
In particular, for every prime ideal $\mpr\not\in\ram(A)$, $\gamma$ is $\SL_2(k_{\mpr})$-conjugate to an element of $\SL_2(\mo_p)$. This will be important in the proof of Lemma~\ref{invariantorbitalintegral} below, when we apply Proposition~\ref{proposition:orbital_integral_in_nonarch} in conjunction with Lemma~\ref{lem1_nona}.

In order to estimate $\bO(\gamma,f_\AA)$ for $[\gamma]\in W$, we write $f_\AA$ in the fully factorized form
\begin{equation}\label{fAA-factorized}
f_{\AA}=\textstyle{\bigotimes}_v f_v,
\end{equation}
where $f_v$ is as in \S\ref{tf-sec} for $v\not\in\ram(A)$ archimedean, $f_v:=\bchar_{\bG(k_v)}$ for $v\in\ram(A)$ arbitrary, $f_\mpr:=\bchar_{K_{\kappa(\mpr)}(\mn_\mpr)}/\mu_\mpr(K_{\kappa(\mpr)}(\mn_\mpr))$ for $\mpr\mid\mn$, and $f_\mpr:=\bchar_{\SL_2(\mo_\mpr)}$ for
$\mpr\not\in\ram(A)$ satisfying $\mpr\nmid\mn$. Then we have the decomposition
\begin{equation}\label{normalized-IOI} \bO(\gamma,f_\AA)=\prod_v\bO'(\gamma,f_v),\qquad\bO'(\gamma,f_v):={|\Delta(\gamma)|}_v^{1/2}\,\bO(\gamma,f_v).
\end{equation}
Indeed, the normalization factors ${|\Delta(\gamma)|}_v^{1/2}$ cancel out by Artin's product formula \cite[Ch.~IV, Th.~5]{W}.

Further, it will be convenient to bound $\bO(\gamma,f_\AA)$ in terms of $\tr(\gamma)$, hence in accordance with \eqref{w0-def} and \eqref{w1-def}, we introduce the functions $w_{j}^{\mn,\mpr}\colon\mo_{\mpr}\to\RR_{\geq 0}$ as
\begin{align}
\label{w0np-def}w_0^{\mn,\mpr}(x)&:=\begin{cases}
N(\mpr)^{\min(v_{\mpr}(x^2-4)/2,\lfloor v_\mpr(\mn)/2\rfloor)},&x^2-4=\square,\\[2pt]
N(\mpr)^{\lfloor v_\mpr(\mn)/2\rfloor} \bchar_{v_\mpr(x^2-4)\geq v_\mpr(\mn)},&x^2-4\neq\square;
\end{cases}\\[4pt]
\label{w1np-def}w_1^{\mn,\mpr}(x)&:=N(\mpr)^{2 v_\mpr(\mn)}\bchar_{v_\mpr(x^2-4)\geq 2v_\mpr(\mn)}.
\end{align}
Thus, for a regular semisimple $\gamma\in\SL_2(\mo_{\mpr})$, we have $w_j^{\mn,\mpr}(\tr\gamma)=w_j^{v_{\mpr}(\mn)}
(\gamma)$ in the notation of \eqref{w0-def} and \eqref{w1-def}. Note that, when $\mpr\nmid\mn$, $w_j^{\mn,\mpr}(x)=1$ for all $x\in\mo$. We also define the function $w_\kappa^\mn\colon\wmo\to\RR_{\geq 0}$ by
\begin{equation}\label{wkn-def}
w_\kappa^\mn(x):=\prod_{\mpr\mid\mn}w_{\kappa(\mpr)}^{\mn,\mpr}(x_\mpr).
\end{equation}

Our estimate on the invariant orbital integrals in \eqref{normalized-IOI} is as follows. We emphasize that the constant $C\geq 4$ in the statement is absolute, independent of all data.

\begin{lemma}\label{invariantorbitalintegral}
There exists an absolute constant $C\geq 4$ such that the following holds. Let $f_{\AA}=\bigotimes_vf_v$ with $f_v\in C_c(\bG(k_v))$ be as in \eqref{fAA-factorized}, and assume that $[\gamma]\in W$. Then:
\begin{enumerate}[(a)]
\listsep
\item\label{IOA-arch}
For every archimedean place $v$ of $k$, $|\bO'(\gamma,f_v)|\leq C$.
\item\label{IOA-nonarch}
For every non-archimedean place $\mpr$ of $k$,
\[\bO'(\gamma,f_\mpr)\leq C{|2|}_\mpr^{-1}{|\Delta_{k(\gamma)/k}|}_\mpr^{1/2}\cdot
\begin{cases}
w_{\kappa(\mpr)}^{\mn,\mpr}(\tr\gamma),&\qquad\mpr\mid\mn;\\
1,&\qquad\mpr\nmid\mn.
\end{cases}\]
Moreover, if $\mpr\nmid\Delta(\gamma)$, then the leading factor $C{|2|}_\mpr^{-1}{|\Delta_{k(\gamma)/k}|}_\mpr^{1/2}$ can be omitted.
\item\label{IOA-global}
The global orbital integral satisfies
\[|\bO(\gamma,f_{\AA})|\leq C^{[k:\QQ]+\omega(\Delta(\gamma))}|N_{k/\QQ}(\Delta_{k(\gamma)/k})|^{-1/2}w_{\kappa}^{\mn}(\tr\gamma),\]
where $\omega(\Delta(\gamma))$ is the number of distinct prime ideals dividing $\Delta(\gamma)$.
\end{enumerate}
\end{lemma}

\begin{proof} Let $v$ be an archimedean place of $k$. If $v\not\in\ram(A)$, then the bound in \eqref{IOA-arch} follows from Propositions~\ref{nontempered} and \ref{tempered}, where the implied constants are absolute.
If $v\in\ram(A)$, then $\bO'(\gamma,f_v)={|\Delta(\gamma)|}_v^{1/2}<2$, since $\gamma$ lies in $\bG(k_v)\simeq\SU_2(\CC)$. This proves \eqref{IOA-arch}.

Before turning to the proof of \eqref{IOA-nonarch}, we make two initial observations.
\begin{itemize}
\listsep
\item First, ${|\Delta_{k(\gamma)/k}|}_\mpr$ equals $1$ or $N(\mpr)^{-1}{|4|}_\mpr$ depending on whether $\mpr$ is unramified or ramified in $k(\gamma)/k$; this follows from \cite[Ch.~VIII, Prop.~3]{W}, \cite[Ch.~I, Prop.~4]{W}, and the Corollary to \cite[Ch.~VIII, Prop.~6]{W}.
\item Second, $\Delta_{k(\gamma)/k}\mid\Delta(\gamma)$, and hence ${|\Delta(\gamma)|}_\mpr\leq {|\Delta_{k(\gamma)/k}|}_\mpr$ for every $\mpr$, because $\Delta(\gamma)$ is the discriminant of the $\mo$-submodule $\mo+\mo w\subset\mo_{k(\gamma)}$, where $w$ is an eigenvalue of $\gamma$.
\end{itemize}
Now, if $\mpr\not\in\ram(A)$, then the bounds in \eqref{IOA-nonarch} follow from Lemma~\ref{lem1_nona}, Proposition~\ref{proposition:orbital_integral_in_nonarch} (where the implied constant is absolute), and our first observation. On the other hand, if $\mpr\in\ram(A)$, then $\mpr\nmid\mn$, and $\bO'(\gamma,f_\mpr)={|\Delta(\gamma)|}_\mpr^{1/2}\leq{|\Delta_{k(\gamma)/k}|}_\mpr^{1/2}$ follows from our second observation and the fact that $\bG_{\gamma}(k_{\mpr})\bs\bG(k_{\mpr})$ is a probability space.

Part \eqref{IOA-global} follows directly from \eqref{normalized-IOI} and the already established parts \eqref{IOA-arch}--\eqref{IOA-nonarch}, with $2C$ as the absolute constant.
\end{proof}

Lemma~\ref{invariantorbitalintegral} combined with Proposition~\ref{prop:TorusVolume} yields the following estimate on the contribution of each regular conjugacy class to the geometric side of the trace formula:
\begin{equation}\label{gammatermbound}
\mu(\bG_\gamma(k)\bs\bG_\gamma(\AA))\bO(\gamma,f_\AA)\ll_\eps C^{\omega(\Delta(\gamma))}
\Delta_k^{1/2+\eps}|N_{k/\QQ}(\Delta_{k(\gamma)/k})|^\eps w_\kappa^\mn(\tr\gamma),\qquad[\gamma]\in W.
\end{equation}

From now on we assume that ${\bm R}=(R_j)_{j\in S}$ satisfies
\begin{equation}
\label{linear-constraint}
\sum_{j\in S}\rho_j R_j\leq7+\log\cC(\Gamma_\kappa(\mn),\bm T),
\end{equation}
where the right hand side is positive by Proposition~\ref{prop:Borel}. In fact, with our final choice \eqref{R-final-choice}, we will have equality in \eqref{linear-constraint} except when $S=\emptyset$, in which case \eqref{linear-constraint} is vacuous. We may use \eqref{linear-constraint} to consolidate the $\eps$-powers in \eqref{gammatermbound} as follows. The condition $[\gamma]\in W$ implies readily via Lemma~\ref{trDeltabounds} that
\begin{equation}\label{where-trace-is}
\begin{alignedat}{3}
&{|\!\tr(\gamma)|}_j<e^{\rho_j(R_j+2)},\qquad&&{|\Delta(\gamma)|}_j < e^{2\rho_j(R_j+2)},\qquad&&j\in S;\\
&{|\!\tr(\gamma)|}_j<e^{2\rho_j},\qquad&&{|\Delta(\gamma)|}_j < e^{4\rho_j},\qquad&&j\in\{1,\dots,a+b+c\}\setminus S.\\
\end{alignedat}
\end{equation}
The bounds \eqref{where-trace-is} combined with \eqref{linear-constraint} and the observation from the proof of Lemma~\ref{invariantorbitalintegral} that $\Delta_{k(\gamma)/k}\mid\Delta(\gamma)$ show that
\begin{equation}\label{technicalbound1}
|N_{k/\QQ}(\Delta_{k(\gamma)/k})|\leq|N_{k/\QQ}(\Delta(\gamma))|\ll\cC(\Gamma_\kappa(\mn),\bm T)^2.
\end{equation}
Moreover, by Proposition~\ref{prop:Borel},
\begin{equation}\label{technicalbound3}
\Delta_k\ll\cC(\Gamma_\kappa(\mn),\bm T)^{2/3}.
\end{equation}
Combining \eqref{tracebound1}, \eqref{gammatermbound}, \eqref{technicalbound1}, \eqref{technicalbound3}, and Lemma~\ref{twopower2}, we
conclude that
\begin{equation}\label{tracebound2}
\tr Rf_\AA \preccurlyeq \cC(\Gamma_\kappa(\mn),\bm{T})+\Delta_k^{1/2}\sum_{[\gamma]\in W}w_\kappa^\mn(\tr\gamma).
\end{equation}
As before, $X\preccurlyeq Y$ stands for $X\ll_\eps\cC(\Gamma,\bm{T})^\eps Y$.
Our estimate on the geometric side of the trace formula \eqref{tracebound2} should be compared with \eqref{multiplicityviatrace} and \eqref{linear-constraint}.

\subsection{Geometry of numbers}\label{geometryofnumbers-sec}
We may further estimate the right hand side of \eqref{tracebound2} as follows. We have seen in \eqref{where-trace-is} that the traces of $[\gamma]\in W$ lie in $\mo\cap\cR$, where
\begin{equation}\label{R-def}
\cR:=\left\{x\in\AA_\infty:\text{${|x|}_j<e^{\rho_j(R_j+2)}$ for all $j\in S$ and
${|x|}_j<e^{2\rho_j}$ for all $j\in\{1,\dots,a+b+c\}\setminus S$}\right\}.
\end{equation}
For a given conjugacy class $[\gamma]\in W$, there are at most $2^{a+|\!\ram_f(A)|+\omega(\Delta(\gamma))}$ conjugacy classes in $\bG(k)$ with the same trace $\tr(\gamma)$, as follows from the Skolem--Noether theorem \cite[Cor.~2.9.9]{MaRe03} and Lemma~\ref{lem:ConjCount}. Moreover, $2^{|\!\ram_f(A)|+\omega(\Delta(\gamma))}\preccurlyeq 1$ by Lemmata~\ref{twopower1}--\ref{twopower2} and \eqref{technicalbound1}. Therefore, by counting how many times a trace appears in \eqref{tracebound2}, we arrive at
\begin{equation}\label{eq:IregEst2}
\tr Rf_\AA\preccurlyeq \cC(\Gamma_\kappa(\mn),\bm{T})+\Delta_k^{1/2}\sum_{x\in\mo\cap\cR}w_\kappa^{\mn}(x).
\end{equation}

We claim that the function $w_\kappa^\mn$ is periodic by the ideal (embedded into $\wmo$)
\begin{equation}\label{n'-def}
\mn':=4\prod_{\mpr\mid\mn}\mpr^{(1+\kappa(\mpr))v_\mpr(\mn)}.
\end{equation}
By \eqref{wkn-def}, it suffices to show for every $\mpr\mid\mn$ that $w_0^{\mn,\mpr}$ is periodic modulo $4\mpr^{v_\mpr(\mn_\mpr)}$ and $w_1^{\mn,\mpr}$ is periodic modulo $\mpr^{2v_\mpr(\mn_\mpr)}$. For this, we first rewrite the definition \eqref{w0np-def} in the more transparent form
\begin{equation}\label{w0np-def-alt}
w_0^{\mn,\mpr}(x)=N(\mpr)^{v_{\mpr}(x^2-4)/2}\bchar_{v_\mpr(x^2-4)<v_\mpr(\mn)}\bchar_{x^2-4=\square}+N(\mpr)^{\lfloor v_\mpr(\mn)/2\rfloor}\bchar_{v_\mpr(x^2-4)\geq v_\mpr(\mn)}.
\end{equation}
Using also the definition \eqref{w1np-def}, it suffices to check the following three claims for any given $r\in\NN$:
\begin{itemize}
\listsep
\item for $x\in\mo_\mpr$, the condition $v_{\mpr}(x^2-4)\geq r$ is invariant under shifts by elements of $\mpr^r$;
\item for $x\in\mo_\mpr$ satisfying $v_{\mpr}(x^2-4)<r$, the quantity $v_{\mpr}(x^2-4)$ is periodic by $\mpr^r$;
\item for $x\in\mo_\mpr$ satisfying $v_{\mpr}(x^2-4)<r$, the condition $x^2-4=\square$ is invariant under shifts by elements of $4\mpr^r$.
\end{itemize}
The first two claims are clear, so we focus on the third one. For $\mpr\nmid 2$, the required invariance follows from the fact that the square map is an automorphism of $1+\mpr\mo_\mpr$ (see \cite[Ch.~II, Prop.~8]{W}). For $\mpr\mid 2$, the required invariance follows from the fact that the square map is an isomorphism from $1+4\mo_\mpr$ to $1+8\mo_\mpr$ (see the proof of \cite[Ch.~II, Prop.~9]{W}).

Now we can estimate the right hand side of \eqref{eq:IregEst2} with the help of the following lemma.
\begin{lemma}\label{lem:CylinderSums}
Let $w\colon\wmo\to\RR_{\geq 0}$ be a continuous function periodic by a nonzero ideal $\fm\subset\mo$, and let $\cP\subset\AA_\infty$ be a polycylinder of the form
\[\cP=\left\{x\in\AA_\infty:\text{${|x_v|}_v\leq P_v$ for every archimedean place $v$ of $k$}\right\}.\]
Then,
\[\sum_{x\in\mo\cap\cP}w(x)\ll\left(\Delta_k^{1/2}N(\fm)+\Delta_k^{-1/2}\vol(\cP)\right)\int_{\wmo} w.\]
\end{lemma}
\begin{proof}
By the periodicity condition, $w$ is a linear combination (with unique nonnegative coefficients) of the characteristic functions of the cosets of
$\fm\wmo$ in $\wmo$. Therefore, by linearity, it suffices to verify the inequality when $w$ is one of these characteristic functions, in which case it states that
\[\#\bigl((y+\fm)\cap\cP\bigr)\ll\Delta_k^{1/2}+\frac{\vol(\cP)}{\Delta_k^{1/2}N(\fm)},\qquad y\in\mo.\]
The claim is trivial when $(y+\fm)\cap\cP=\emptyset$. Otherwise, let us fix a point $x_0\in(y+\fm)\cap\cP$. For any $x\in(y+\fm)\cap\cP$, we have $x-x_0\in\fm\cap(2\cP)$, hence it suffices to show that
\[\#\bigl(\fm\cap(2\cP)\bigr)\ll\Delta_k^{1/2}+\frac{\vol(\cP)}{\Delta_k^{1/2}N(\fm)}.\]
This inequality is equivalent to \cite[Cor.~1]{FHM20}\footnote{Strictly speaking, \cite{FHM20} is written for $k\neq\QQ$, but the conclusions extend trivially to $k=\QQ$.}, hence we are done.
\end{proof}
Combining \eqref{n'-def} with Proposition~\ref{prop:Borel} and Lemma~\ref{grouplemma2}, we verify that
\begin{equation}\label{volRbound}
N(\mn')\ll\prod_{\mpr\mid\mn}N(\mpr)^{(1+\kappa(\mpr))v_\mpr(\mn)}\ll\Delta_k^{-3/2}\cC(\Gamma_\kappa(\mn),\bm T);
\end{equation}
moreover, \eqref{R-def} and \eqref{linear-constraint} yield $\vol(\cR)\ll\cC(\Gamma_{\kappa}(\mn),\bm{T})$. Therefore, Lemma~\ref{lem:CylinderSums} furnishes the clean bound
\begin{equation}\label{wkn-sum}
\sum_{x\in\mo\cap\cR}w_\kappa^{\mn}(x)\ll\Delta_k^{-1/2}\,\cC(\Gamma_\kappa(\mn),\bm T)\int_{\wmo}w_\kappa^{\mn}.
\end{equation}
The integral over $\wmo$ splits into a product of local integrals as
\[\int_{\wmo}w_\kappa^{\mn}=\prod_{\mpr\mid\mn}\int_{\mo_\mpr}w_{\kappa(\mpr)}^{\mn,\mpr},\]
and the factors can be estimated as (cf.\ \eqref{w0np-def-alt} and \eqref{w1np-def})
\[\int_{\mo_\mpr}w_0^{\mn,\mpr}\leq\int_{\mo_\mpr}\Biggl(\;
\sum_{0\leq\ell<v_\mpr(\mn)/2}N(\mpr)^\ell\bchar_{v_\mpr(x^2-4)=2\ell}+
N(\mpr)^{\lfloor v_\mpr(\mn)/2\rfloor}\bchar_{v_\mpr(x^2-4)\geq v_\mpr(\mn)}\Biggr)dx\leq
\begin{cases}1,&\mpr\nmid 2,\\2{|4|}_\mpr^{-1},&\mpr\mid 2;\end{cases}\]
\[\int_{\mo_\mpr}w_1^{\mn,\mpr}=\int_{\mo_\mpr}N(\mpr)^{2 v_\mpr(\mn)}\bchar_{v_\mpr(x^2-4)\geq 2v_\mpr(\mn)}\,dx\leq
\begin{cases}2,&\mpr\nmid 2,\\4{|4|}_\mpr^{-1},&\mpr\mid 2.\end{cases}\]
Multiplying these local bounds over $\mpr\mid\mn$, we get by Lemma~\ref{twopower2} and \eqref{volRbound} that
\begin{equation}\label{wkn-int}
\int_{\wmo}w_{\kappa}^{\mn}\leq2^{\omega(\mn)}8^{[k:\QQ]}\preccurlyeq 1.
\end{equation}

Combining \eqref{eq:IregEst2}, \eqref{wkn-sum}, \eqref{wkn-int}, we see that $\tr Rf_\AA\preccurlyeq\cC(\Gamma_\kappa(\mn),\bm T)$.
Going back to \eqref{multiplicityviatrace}, we have proved that
\begin{equation}
\label{almost-done}
\sum_{\pi\in\cB(\bm\sigma,\bm T)}\mm(\pi,\Gamma_\kappa(\mn))\preccurlyeq
e^{-\sum_{j\in S}2\rho_jR_j\sigma_j}\,\cC(\Gamma_\kappa(\mn),\bm T).
\end{equation}
This estimate holds for every choice of nonnegative parameters $\bm{R}=(R_j)_{j\in S}$ satisfying \eqref{linear-constraint}. If $S=\emptyset$, this proves Theorem~\ref{thm:Main}. In fact, it is not hard to see that, for $S=\emptyset$, the proof gives \eqref{almost-done} in the slightly stronger form
\begin{equation}\label{slightly-stronger}
\sum_{\pi\in\cB(\bm\sigma,\bm T)}\mm(\pi,\Gamma_{\kappa}(\mn))\ll_{\eps}\vol(\Gamma_{\kappa}(\mn)\bs G)^{\eps}\cC(\Gamma_{\kappa}(\mn),\bm T).
\end{equation}

Otherwise, let $j_0\in S$ be such that $\sigma_{j_0}$ is maximal, that is, $p({\bm\sigma})=2/(1-2\sigma_{j_0})$ by \eqref{p-lambda}. Then, by setting
\begin{equation}
\label{R-final-choice}
R_j:=\begin{cases}\rho_{j_0}^{-1}(7+\log\cC(\Gamma_\kappa(\mn),\bm T)),&\text{if $j=j_0$,}\\0,&\text{if $j\neq j_0$,}\end{cases}
\end{equation}
we get
\[\sum_{\pi\in\cB(\bm \sigma,\bm T)}\mm(\pi,\Gamma_\kappa(\mn))\ll_\eps\cC(\Gamma_\kappa(\mn),\bm T)^{1-2\sigma_{j_0}+\eps}=
\cC(\Gamma_\kappa(\mn),\bm T)^{2/p{(\bm\sigma)}+\eps}.\]
The proof of Theorem~\ref{thm:Main} is complete.

\end{document}